\documentclass{m2an}
\allowdisplaybreaks
\usepackage{amssymb}
\usepackage{amsmath,bm}
\usepackage{multirow}
\usepackage{makecell}
\usepackage{graphicx}
\usepackage{subfigure}
\usepackage{float}
\usepackage{lipsum}
\usepackage{amsfonts}
\usepackage{graphicx}
\usepackage{epstopdf}
\usepackage{algorithmic}
\usepackage{amsopn}
\usepackage{amssymb,amsmath}
\usepackage{amsfonts}
\usepackage{url}
\usepackage[colorlinks,linktocpage,linkcolor=blue]{hyperref}
\usepackage{algorithm,algorithmic}
\usepackage{dsfont,verbatim}
\usepackage{enumitem}
\usepackage{mathrsfs}
\usepackage{epsfig}
\usepackage{graphicx}
\usepackage{color}
\usepackage{empheq}
\usepackage{mathabx}
\usepackage{epstopdf}

\newtheorem{theorem}{Theorem}[section]
\newtheorem{lemma}[theorem]{Lemma}

\theoremstyle{definition}

\newtheorem{example}[theorem]{Example}

\theoremstyle{remark}
\newtheorem{remark}[theorem]{Remark}

\numberwithin{equation}{section}
\numberwithin{table}{section}

\def\d{\mbox{\boldmath $d$}}
     
    \def\f{\mbox{\boldmath $f$}}

    \def\H{\mbox{\boldmath $H$}}
     \def\n{\mbox{\boldmath $n$}}
    \def\u{\mbox{\boldmath $u$}}
    
    \def\J{\mbox{\boldmath $J$}}

    \def\v{\mbox{\boldmath $v$}}
    \def\w{\mbox{\boldmath $w$}}
    \def\l{\mbox{\boldmath $l$}}

     \def\0{\mbox{\boldmath $0$}}

    \def\u{{\bm u}}
\def\v{{\bm v}}
\def\H{{\bm H}}
\def\J{{\bm J}}

\def\d{{\rm d}}
\def\R{{\bm R}}

\begin{document}

\title{Optimal error estimates of a second-order  projection finite element method for magnetohydrodynamic equations}

\author{Cheng Wang}\address{Mathematics Department, University of Massachusetts, North Dartmouth, MA 02747, USA.  \\
    \email{cwang1@umassd.edu}.
    The first author's research is supported in part by NSF DMS-2012669. 	}

\author{Jilu Wang}\address{Beijing Computational Science Research Center, Beijing 100193, China.
	 \email{jiluwang@csrc.ac.cn}.
    The second author's research is supported in part by NSFC-U1930402 and NSFC-12071020}

\author{Zeyu Xia}\address{School of Mathematics Sciences,	University of Electronic Science and Technology of China,
	Chengdu 611731, China.
    \email{zeyuxia@std.uestc.edu.cn}. The third author's research is supported in part by NSFC-11871139.}

\author{Liwei Xu}\address{School of Mathematics Sciences,	University of Electronic Science and Technology of China,
	Chengdu 611731, China.  Corresponding author: \email{xul@uestc.edu.cn}. The fourth author's research is supported in part by NSFC-11771068 and NSFC-12071060.}

\begin{abstract}
In this paper, we propose and analyze a temporally second-order accurate, fully discrete finite element method for the magnetohydrodynamic (MHD) equations. A modified Crank--Nicolson method is used to discretize the model and appropriate semi-implicit treatments are applied to the fluid convection term and two coupling terms. These semi-implicit approximations result in a linear system with variable coefficients for which the unique solvability can be proved theoretically. In addition, we use a decoupling projection method of the Van Kan type \cite{vankan1986} in the Stokes solver, which computes the intermediate velocity field based on the gradient of the pressure from the previous time level, and enforces the incompressibility constraint via the Helmholtz decomposition of the intermediate velocity field. The energy stability of the scheme is theoretically proved, in which the decoupled Stokes solver needs to be analyzed in details.
Optimal-order convergence of $\mathcal{O} (\tau^2+h^{r+1})$ in the discrete $L^\infty(0,T;L^2)$ norm is proved for the proposed decoupled projection finite element scheme, where $\tau$ and $h$ are the time stepsize and spatial mesh size, respectively, and $r$ is the degree of the finite elements.
Existing error estimates of second-order projection methods of the Van Kan type \cite{vankan1986} were only established in the discrete $L^2(0,T;L^2)$ norm for the Navier--Stokes equations.
Numerical examples are provided to illustrate the theoretical results.
\end{abstract}


\maketitle

\section*{Introduction}

The magnetohydrodynamic equations have been widely applied into metallurgy and liquid-metal processing, and the numerical solutions are of great significance in practical scientific and engineering applications; see~\cite{app1} and~\cite{app2}. Such an MHD system could be formulated as \cite{Shercliff1965}
\begin{align}
&    \mu\partial_t\H+\sigma^{-1}\nabla\times(\nabla\times\H)
        -\mu\nabla\times(\u\times\H)=\sigma^{-1}\nabla\times {\J},
        \label{PDEa}\\
&
    \partial_t\u+\u\cdot\nabla\u-\nu\Delta\u+\nabla p
        =\f-\mu\H\times(\nabla\times\H),
        \label{PDEb}\\
&
    \nabla\cdot\u=0,\label{PDEc}
\end{align}
over $\Omega\times(0,T]$, where $\Omega$ is a bounded and convex polyhedral domain in $\mathbb{R}^3$ (polygonal domain in $\mathbb{R}^2$). In the above system, $\u$, $\H$ and $p$ denote the velocity field, the magnetic filed, and the pressure, respectively; ${\J}$ and $\f$ are the given source terms (${\J}$ denotes a scalar function in $\mathbb{R}^2$); $\sigma$ denotes the magnetic Reynolds number, $\nu$ the viscosity of the fluid   , and $\mu=M^2\nu\sigma^{-1}$, where $M$ is the Hartman number. The initial data and boundary conditions are given by
\begin{align}
&\H|_{t=0}=\H_0,\quad\u|_{t=0}=\u_0,
\quad\quad
\mbox{in }\Omega,  \label{initial data-1}\\
&
  \H\times\n={\bf{0}}, \quad\ \
    \u={\bf{0}},
\qquad\quad\quad \,
\mbox{on }\partial\Omega\times(0,T] . \label{BC-1}
\end{align}
It is assumed that the initial data satisfies
\begin{equation}
\nabla\cdot \H_0=\nabla\cdot\u_0=0.
  \label{regularity_source}
\end{equation}
By taking the divergence of \eqref{PDEa}, one can easily get $\mu\partial_t\nabla\cdot\H=0$, which together with the above divergence-free initial condition
implies that
\begin{align}\label{div-H=0}
\nabla\cdot\H=0.
\end{align}
The existence and uniqueness of the weak solution for this problem has been theoretically proved in \cite{Max1991,PDE1}. More regularity analysis of the MHD system could be referred in~\cite{PDE7,PDE5,PDE6,PDE4}, etc.

There have been many existing works on the numerical approximations for the incompressible MHD system. In bounded and convex domains, the solutions of the MHD model are generally in $[H^1(\Omega)]^d$ ($d=2,3,$ denotes the dimension of $\Omega$) and therefore people often use $H^1$-conforming finite element methods (FEMs) to solve the MHD equations numerically.  For example, Gunzburger, Meir, and Peterson \cite{Max1991} used $H^1$-conforming FEMs for solving the stationary incompressible MHD equations with an optimal error estimate being established. Later, He developed $H^1$-conforming FEMs in \cite{FEMr3} for solving the time-dependent MHD equations and proved error estimates of the numerical scheme. More works on $H^1$-conforming FEMs for the MHD equations can be found in
\cite{FEMr1,Gerbeau2000A,FEMr2,FEMr4,main}.

While the spatial approximation has always been important, the temporal discretization also plays a significant role for solving the MHD system. There have been quite a few existing stability and convergence analyses for the first-order temporally accurate numerical schemes~\cite{first,FEMr3,FDM1,ratecon,FDM2}.  In most of these works, the stability and convergence analyses have been based on a Stokes solver at each time step, i.e., the computation of the pressure gradient has to be implemented with the incompressibility constraint being enforced, which in turn leads to a non-symmetric linear system, and the computation costs turn out to be extremely expensive. To overcome this difficulty, some ``decoupled" techniques have been introduced. In \cite{shen2016A}, Zhao, Yang, Shen, and Wang dealt with a binary hydrodynamic phase field model of mixtures of nematic liquid crystals and viscous fluids by designing a decoupled semi-discrete scheme, which is linear, first-order accurate in time, and unconditionally energy stable. In particular, a pressure-correction scheme \cite{Guermond2006} was used so that the pressure could be explicitly updated in the velocity equation by introducing an intermediate function and thus two sub-systems are generated. In \cite{shen2015}, Liu, Shen, and Yang proposed a first-order decoupled scheme for a phase-field model of two-phase incompressible flows with variable density based on a ``pressure-stabilized" formulation, which treats the pressure term explicitly in the velocity field equation, and only requires a Poisson solver to update the pressure. These works have mainly focused on the design of energy-preserving schemes without presenting the convergence analysis. Meanwhile, the first-order temporal accuracy may not be sufficient in the practical computations of the MHD system, and therefore higher-order temporal numerical approximations have been highly desired.

In the development of temporally higher-order methods, a conditionally stable second-order backward difference formula (BDF2) algorithm was proposed in \cite{FEM3} for a reduced MHD model at small magnetic Reynolds number, in which the coupling terms were explicitly updated, and other terms were implicitly computed.
An unconditionally stable BDF2 method was proposed in \cite{Heister2016}, where the method was proved convergent with optimal order.
In~\cite{second2}, a second-order scheme with Newton treatment of the nonlinear terms was proposed, where the unconditional stability and optimal error estimates were obtained. Recently, a fully discrete Crank--Nicolson (CN) scheme was studied in
\cite{FEMr5}, where the  unconditional energy stability and convergence (without error estimates) were proved.
For efficient large scale numerical simulations of incompressible flows, high-order projection methods are desired.
In \cite{Shen1996}, Shen presented rigorous error analysis of second-order Crank--Nicolson projection methods of the Van Kan type \cite{vankan1986}, i.e., second-order incremental pressure-correction methods, for the unsteady incompressible Navier--Stokes equations. By interpreting the respective projections schemes as second-order time discretizations of a perturbed system which approximates the Navier--Stokes equations, optimal-order convergence in the discrete $L^2(0,T;L^2)$ norm was proved for the semi-implicit schemes.
Later, Guermond \cite{Guermond1999} proved optimal error estimates in the discrete $L^2(0,T;L^2)$ norm for the fully discrete case with BDF2 approximation in time.
However, whether second-order incremental pressure-correction methods have optimal convergence in the discrete $L^\infty(0,T;L^2)$ norm remains open for both Navier--Stokes and MHD equations.

In this article, we fill in the gap between numerical computation and rigorous error estimates for a second-order projection finite element method for the MHD model. 
We first propose a temporally second-order accurate, fully discrete decoupled finite element method for the MHD system~\eqref{PDEa}-\eqref{PDEc}, and then the following properties are theoretically established: unique solvability, unconditional energy stability, and optimal-rate convergence analysis. In particular, a modified Crank--Nicolson method with an implicit Adams--Moulton interpolation in the form of $\frac34\H_h^{n+1}+\frac14\H_h^{n-1}$, instead of the standard Crank--Nicolson approximation, is applied to discretize the magnetic diffusion term. Such a technique leads to a stronger stability property of the numerical scheme, as will be demonstrated in the subsequent analysis.
A second-order incremental pressure-correction method is used to decouple the computation of velocity and pressure.
Precisely, an intermediate velocity function $\widehat\u_h^{n+1}$ is introduced in the numerical scheme, and its computation is based on the pressure gradient at the previous time step. After solving the intermediate velocity field, we decompose it into the divergence-free subspace by using the Helmholtz decomposition. This yields the velocity field $\u_h^{n+1}$ at the same time level. 
In the error analysis, we first introduce an intermediate projection of the velocity, i.e., $\widehat{\R_h\u^{n+1}}$ as introduced in~\eqref{Rh-hat}-\eqref{Rh-hat2}, with which, estimate of an intermediate error for the velocity is obtained. With such estimate and rigorous analysis of the discrete gradient of the Stokes projection, optimal convergence $\mathcal{O}(\tau^2+h^{r+1})$ is proved for the velocity and magnetic fields in the discrete $L^\infty (0, T; L^2)$ norm, where $r$ is the degree of the finite elements, $\tau$ and $h$ are the time stepsize and space mesh size, respectively.
To our knowledge, this is the first rigorous analysis of optimal-order convergence in the discrete $L^\infty (0, T; L^2)$ norm for second-order incremental pressure-correction methods.
The techniques introduced in this paper would also work for other related projection methods.


This paper is organized as follows. In Section~\ref{sec:formulation}, a variational formulation and some preliminary results are reviewed. The fully discrete finite element scheme is introduced in Section~\ref{sec:numerical scheme}, and its unconditional energy stability is established in details. Section \ref{sec-error} provides the rigorous proof of the unique solvability and optimal error estimates. Several numerical examples are presented in Section~\ref{sec:numerical results}. Finally, some concluding remarks are provided in Section~\ref{sec:conclusion}.

\section{Variational formulation and stability analysis} \label{sec:formulation}
\setcounter{equation}{0}
For $k\ge 0 $ and $1\le p\le \infty$, let $W^{k,p}(\Omega)$ be the conventional Sobolev space of functions defined on $\Omega$, with abbreviations $L^{p}(\Omega)=W^{0,p}(\Omega)$ and $H^{k}(\Omega)=W^{k,2}(\Omega)$. Then, we denote by $W^{1,p}_0(\Omega)$ the space of functions in $W^{1,p}(\Omega)$ with zero traces on the boundary $\partial\Omega$, and denote $H^1_0(\Omega)=W^{1,2}_0(\Omega)$. The corresponding vector-valued spaces are
\begin{align*}
&{\bf L}^p(\Omega)=[L^p(\Omega)]^d ,
\qquad\qquad\qquad {\bf W}^{k,p}(\Omega)=[W^{k,p}(\Omega)]^d , \\
&{\bf W}^{1,p}_0(\Omega)=[W^{1,p}_0(\Omega)]^d ,
\qquad\quad\ \ \ {\bf H}^1_0(\Omega)={\bf W}^{1,2}_0(\Omega) , \\
&\ring{\bf H}^1(\Omega)=\{\v\in{\bf H}^1(\Omega):\v\times\n|_{\partial\Omega}=0 \},
\end{align*}
where $d=2,3,$ denotes the dimension of $\Omega$.
As usual, the inner product of $L^2(\Omega)$ is denoted by $(\cdot,\cdot)$.

With the above notations, it could be seen that the exact solution $(\H,\u,p)$ of \eqref{PDEa}-\eqref{PDEc} satisfies
\begin{align}
&    (\mu\partial_t\H,\w)+(\sigma^{-1}\nabla\times\H,\nabla\times\w)
        -(\mu\u\times\H,\nabla\times\w)=(\sigma^{-1}\nabla\times \J,\w), \label{weaksolution-1}\\
&   (\partial_t\u,\v)+(\nu\nabla\u,\nabla\v)+b(\u,\u,\v) -(p,\nabla\cdot\v)
= (\f,\v)-(\mu\H\times(\nabla\times\H),\v), \label{weaksolution-2}\\
&(\nabla\cdot\u,q)=0,
    \label{weaksolution-3}
\end{align}
for any test functions $(\w,\v,q)\in( \ring{\bf H}^1(\Omega), {\bf H}^1_0(\Omega), L^2(\Omega))$, where we have defined the trilinear form $b(\cdot,\cdot,\cdot)$ as
\begin{equation}
\begin{aligned}
b(\u,\v,\w)&:=(\u\cdot\nabla\v,\w)+\frac12((\nabla\cdot\u)\v,\w)  \\
&=\frac12\big[(\u\cdot\nabla\v,\w) - (\u\cdot\nabla\w,\v)\big],
\quad \forall\u,\v,\w\in{\bf H}_0^1(\Omega) ,
\end{aligned}
\label{def-b}
\end{equation}
and $\u\cdot\v$ denotes the Euclidean scalar product in $\mathbb{R}^d$. Notice that the trilinear form $b(\cdot,\cdot,\cdot)$ is skew-symmetric with respect to its last two arguments,  so that we further have
\begin{align*}
b(\u,\v,\v)=0,
\quad
\forall\u, \v,\w\in{\bf H}_0^1(\Omega).
\end{align*}

The energy stability of the continuous system \eqref{weaksolution-1}-\eqref{weaksolution-3} could be obtained in a straightforward manner.
By taking $\w=\H,\v=\u$ in \eqref{weaksolution-1}-\eqref{weaksolution-3} and adding the resulting equations together, we get
\begin{equation*}
 \begin{aligned}
    &
    \frac{\mu}{2}\frac{\d}{\d t}\|\H\|_{L^2}^2 + \sigma^{-1}\|\nabla\times\H\|_{L^2}^2 + \frac{1}{2}\frac{\d}{\d t}\|\u\|_{L^2}^2 +
    \nu\|\nabla\u\|_{L^2}^2 \\
    &= (\J,\sigma^{-1}\nabla\times\H)+(\f,\u) \\
    &\le \frac{1}{4\sigma}\|\J\|_{L^2}^2+\sigma^{-1}\|\nabla\times\H\|_{L^2}^2
        +\frac{1}{4\varepsilon}\|\f\|_{L^2}^2+\varepsilon\|\u\|_{L^2}^2 ,
 \end{aligned}
\end{equation*}
where $\varepsilon$ is an arbitrary constant. Due to the zero boundary condition of $\u$ in \eqref{BC-1}, we have $\|\u\|^2\leq C\|\nabla\u\|^2$. Since $\varepsilon$ can be arbitrarily small, we obtain the following energy estimate
\begin{equation}
    \frac{\mu}{2}\frac{\d}{\d t}\|\H\|_{L^2}^2  + \frac{1}{2}\frac{\d}{\d t}\|\u\|_{L^2}^2 \\
    \leq \frac{1}{4\sigma}\|\J\|_{L^2}^2+\frac1{4\varepsilon}\|\f\|_{L^2}^2.
\end{equation}
If the sources terms $\J={\bm f}={\bm 0}$, we further get
\begin{align}
 \frac{\mu}{2}\frac{\d}{\d t}\|\H\|_{L^2}^2  + \frac{1}{2}\frac{\d}{\d t}\|\u\|_{L^2}^2
 \leq0,
\end{align}
which implies the total energy is decaying.

\section{Numerical method and theoretical results} \label{sec:numerical scheme}

\subsection{Numerical method}
In this subsection, we propose a fully discrete decoupled finite element method for solving the system \eqref{PDEa}-\eqref{PDEc}. Let $\Im_h$ denote a quasi-uniform partition of $\Omega$ into tetrahedrons $K_j$ in $\mathbb{R}^3$ (or triangles in $\mathbb{R}^2$), $j=1,2,\dots, M$, with mesh size $h=\max_{1\le j\le 	M}\{\mbox{diam}K_j\}$.
To approximate $\u$ and $p$ in the system \eqref{PDEa}-\eqref{PDEc}, we introduce the Taylor-Hood finite element space ${\bf X_h}\times M_h$, defined by
\begin{align*}
      &{\bf X}_h=\{\l_h\in{\bf H}_0^1(\Omega):\l_h|_{K_j}\in{\bf P}_{r}(K_j)\}, \\
    &M_h=\{q_h\in L^2(\Omega):q_h|_{K_j}\in P_{r-1}(K_j),\ \mbox{$\int_\Omega q_h \d x=0$}\},
\end{align*}
for any integer $r\ge 2$, where $P_r(K_j)$ is the space of polynomials with degree $r$ on $K_j$ for all $K_j\in\Im_h$ and ${\bf P}_{r}(K_j):=[P_r(K_j)]^d$. To approximate the magnetic field $\H$, we introduce the finite element space ${\bf S}_h$ defined by
\begin{align*}
&{\bf S}_h=\{\w_h\in \ring{\bf H}^1(\Omega):\w_h|_{K_j}
        \in {\bf P}_{r}(K_j)\} .
\end{align*}

Let $\{t_n=n\tau\}_{n=0}^N$ denote a uniform partition of the time interval $[0,T]$,
with a step size $\tau=T/N$, and $v^n=v(x,t_n)$.
For any sequences $\{v^n\}_{n=0}^N$ and $\{\widehat v^n\}_{n=0}^N$, we define
\begin{align*}
\widecheck v^{n+\frac12}:=\frac34 v^{n+1}+\frac14 v^{n-1}, \quad
\overline v^{n+\frac12}:=\frac12 \widehat v^{n+1}+\frac12 v^{n}, \quad
\widetilde v^{n+\frac12}:=\frac32 v^{n}-\frac12 v^{n-1}.
\end{align*}
Then, a  fully discrete second-order decoupled FEM for the incompressible MHD equations \eqref{PDEa}-\eqref{PDEc} is formulated as: find
$(\H_h^{n+1},\u_h^{n+1},\widehat{\u}_h^{n+1},p_h^{n+1})\in ({\bf S}_h,{\bf X}_h,{\bf X}_h,M_h)$ such that
\begin{align}
&
\mu\bigg(\frac{\H^{n+1}_h-\H^n_h}{\tau},\w_h\bigg)
+\sigma^{-1}\Big(\nabla\times\widecheck\H_h^{n+\frac12},\nabla\times\w_h \Big)
+\sigma^{-1}\Big(\nabla\cdot\widecheck\H_h^{n+\frac12},\nabla\cdot\w_h \Big) \nonumber\\
&\hspace{1.7in}
-\mu\Big(\overline\u_h^{n+\frac12}\times\widetilde\H_h^{n+\frac12} ,\nabla\times\w_h\Big)
=\sigma^{-1}\Big(\nabla\times \J^{n+\frac12},\w_h\Big) ,  \label{scheme-a} \\
&
\bigg(\frac{\widehat{\u}_h^{n+1}-\u_h^n}{\tau},\v_h\bigg)
+\nu\Big(\nabla\overline\u_h^{n+\frac12},\nabla\v_h\Big)
+b\Big(\widetilde\u_h^{n+\frac12}, \overline\u_h^{n+\frac12},\v_h \Big)
-\Big(p_h^n,\nabla\cdot\v_h\Big) \nonumber\\
&\hspace{2.1in}
+\mu\Big(\widetilde\H_h^{n+\frac12}\times(\nabla\times\widecheck\H_h^{n+\frac12}),\v_h\Big)
=\Big(\f^{n+\frac12},\v_h\Big), \label{scheme-b} \\
&\bigg(\frac{{\u}_h^{n+1}-\widehat{\u}_h^{n+1}}{\tau},\l_h\bigg)-\frac12\Big(p_h^{n+1}-p_h^n,\nabla\cdot\l_h\Big)=0,  \label{scheme-c} \\
& \Big(\nabla\cdot\u_h^{n+1},q_h\Big)=0,   \label{scheme-d}
\end{align}
for any $(\w_h,\v_h,\l_h,q_h) \in ({\bf S}_h,{\bf X}_h,{\bf X}_h,M_h)$ and $n=1,2,\dots,N-1$.
Here we have added a stabilization term $\sigma^{-1}(\nabla\cdot\widecheck\H_h^{n+\frac12},\nabla\cdot\w_h)$ to \eqref{scheme-a}. This is consistent with \eqref{weaksolution-1} in view of \eqref{div-H=0}.

In this paper, it is assumed that the system \eqref{PDEa}-\eqref{PDEc} admits a unique solution
satisfying
\begin{align}
&\|\H_{ttt}\|_{L^\infty(0,T;L^2)}
+\|\H_{tt}\|_{L^\infty(0,T;H^1)}
+\|\H_t\|_{L^\infty(0,T;H^{r+1})}
+\|\u_{ttt}\|_{L^\infty(0,T;L^2)}
\nonumber\\
&\quad
+\|\u_{tt}\|_{L^\infty(0,T;H^1)}
+\|\u_t\|_{L^\infty(0,T;H^{r+1})}
+\|p_{tt}\|_{L^\infty(0,T;L^2)}
+\|p_{t}\|_{L^\infty(0,T;H^r)}
\le
K .
\label{reg-asp}
\end{align}
Here, the subscripts of $\H,\u,p$ denote the partial derivative to variable $t$.

Next, we present our main results, i.e., optimal error estimates for scheme \eqref{scheme-a}-\eqref{scheme-d}, in the following theorem.

\begin{theorem}\label{thm-error}
Suppose that the system \eqref{PDEa}-\eqref{PDEc} has a unique solution $(\H,\u,p)$ satisfying \eqref{reg-asp}. Then there exist positive constants $\tau_0$ and $h_0$ such that when $\tau<\tau_0$, $h<h_0$, and $\tau=\mathcal{O}(h)$, the fully discrete decoupled FEM system \eqref{scheme-a}-\eqref{scheme-d} admits a unique solution $(\H_h^n,\u_h^n,p_h^n)$, $n=2,3,\dots,N$, which satisfies that
\begin{align}
&\max_{2\le n\le N}\Big(\|\H_h^n-\H^n\|_{L^2}+\|\u_h^n-\u^n\|_{L^2} \Big)
\le
C_0(\tau^2+h^{r+1}), \label{est-error}\\
&
\bigg( \tau\sum_{n=2}^{N} \big(\|\nabla\times(\H_h^{n}-\H^{n})\|_{L^2}^2+  \|\nabla(\overline\u_h^{n-\frac12}-\overline\u^{n-\frac12})\|_{L^2}^2 \big)\bigg)^\frac12
\le
C_0(\tau^2+h^{r}),
\label{est-error2}
\end{align}
where $C_0$ is a positive constant independent of $\tau$ and $h$.
\end{theorem}

\vspace{.1in}
\begin{remark}
\rm
One feature of the proposed numerical scheme~\eqref{scheme-a}-\eqref{scheme-d} is associated with its decoupled nature in the Stokes solver.
Motivated by the second-order projection method of the Van Kan type \cite{vankan1986}, i.e., second-order incremental pressure-correction method,
we introduce an intermediate velocity $\widehat{\u}_h^{n+1}$ to decouple the problem, and thus build two systems and both of them consist of two unknowns. More precisely, we first obtain $\H^{n+1}_h$ and $\widehat{\u}^{n+1}_h$ through \eqref{scheme-a}-\eqref{scheme-b}, while treating the gradient of pressure explicitly. Then, we substitute $\widehat{\u}_h^{n+1}$ into~\eqref{scheme-c}-\eqref{scheme-d}, so that $p_h^{n+1}$ and $\u_h^{n+1}$ could be efficiently computed via solving a Darcy problem. In comparison with the classical coupled solver that the full system contains three unknowns $\H_h^{n+1}$, $\u_h^{n+1}$ and $p_h^{n+1}$, which have to be solved simultaneously, such a decoupled approach will greatly improve the efficiency of the numerical scheme.

There have been extensive analyses of decoupled numerical schemes for incompressible Navier--Stokes equations; see the pioneering works of A.~Chorin~\cite{Chorin1968}, R.~Temam~\cite{Temam1969}, and many other related studies~\cite{BCG1989,E1995,E2002,Kim1985,STWW2003,Wang2000}, etc.
In all the existing articles, error estimates of second-order projection methods of the Van Kan type \cite{vankan1986} were only established in the discrete $L^2 (0, T; L^2)$ norm for the Navier--Stokes equations \cite{Guermond1999,Guermond2006,Shen1996}.
To our knowledge, Theorem \ref{thm-error} of this paper is the first optimal error estimate in the discrete $L^\infty (0, T; L^2)$ norm for second-order projection methods of the Van Kan type \cite{vankan1986}.
The techniques introduced in this paper would also work for other related projection methods.


\end{remark}

\vspace{.05in}
\begin{remark}
\rm
Another feature of scheme \eqref{scheme-a}-\eqref{scheme-d} is that we have used a modified Crank--Nicolson method for temporal discretization, where the term $\nabla\times \H^{n+\frac12}$ is approximated by $\nabla\times(\frac34\H^{n+1}+\frac14\H^{n-1})$.
This enables us to obtain error estimates for the term $\nabla\times\H$ at certain time steps, instead of an average of those at two consecutive time levels; see \eqref{est-error2}.
Such a modified Crank--Nicolson scheme has been extensively applied to various gradient flow models~\cite{chen14,cheng16a,diegel16,guo16}. An application of this approach to the incompressible MHD system is reported in this work, for the first time.
\end{remark}

\vspace{.05in}
\begin{remark}
\rm
In \eqref{scheme-a}, we have added a stabilization term $\sigma^{-1}(\nabla\cdot\widecheck\H_h^{n+\frac12},\nabla\cdot\w_h)$ to validate the coercivity of the magnetic equation, with which, optimal error estimates for the magnetic field in energy-norm can be proved.
\end{remark}

\vspace{.05in}
\begin{remark} \label{H1u1}
\rm
It is noted that the numerical solutions at two previous time levels are needed for the implementation of \eqref{scheme-a}-\eqref{scheme-d}.
The starting values at time steps $t_0$ and $t_1$ are assumed to be given and satisfy the estimates \eqref{est-error}-\eqref{est-error2}.

\end{remark}

\vspace{.05in}
In the following subsection,
we analyze the energy stability of scheme \eqref{scheme-a}-\eqref{scheme-d}.
In this paper, we denote by $C$ a generic positive constant and by $\varepsilon$ a generic small positive constant, which are independent of $n$, $h$, $\tau$, and $C_0$.

\subsection{Stability analysis of numerical scheme}

In this subsection, we present the energy stability analysis for the numerical system \eqref{scheme-a}-\eqref{scheme-d}.
Here, we introduce a discrete version of the gradient operator, $\nabla_h: M_h\rightarrow {\bf X}_h$, defined as
\begin{align}
(\v_h,\nabla_h q_h)=-(\nabla\cdot \v_h,q_h), \label{def-nablah}
\quad
\forall \v_h\in{\bf X}_h, q_h\in M_h.
\end{align}
Through the definition of the discrete gradient operator $\nabla_h$, we can rewrite the equation \eqref{scheme-c} in the following equivalent form:
\begin{align}
&\frac{\u_h^{n+1}-\widehat\u_h^{n+1}}{\tau}
+\frac12\nabla_h(p_h^{n+1}-p_h^n)=0 . \label{scheme-c2}
\end{align}
The abstract form \eqref{scheme-c2} will be useful in the stability analysis of numerical scheme.

\begin{theorem}\label{stability}
The numerical solution $(\H_h^n,\u_h^n,p_h^n)$ to the fully discrete linearized FEM \eqref{scheme-a}-\eqref{scheme-d} satisfies the following energy stability estimate
  \begin{equation}\label{est-stability}
  \begin{aligned}
    \frac\mu{2\tau}\Big(\|\H_h^{n+1}\|_{L^2}^2-\|\H_h^n\|_{L^2}^2\Big)
    +\frac\mu{8\tau}\Big(\|\H_h^{n+1}-\H_h^n\|_{L^2}^2
    -\|\H_h^{n}-\H_h^{n-1}\|_{L^2}^2\Big)  \\
    +\frac1{2\tau}\Big(\|\u_h^{n+1}\|_{L^2}^2-\|\u_h^n\|_{L^2}^2\Big)
    +\frac{\tau}8\Big(\|\nabla_h p_h^{n+1}\|_{L^2}^2-\|\nabla_h p_h^n\|_{L^2}^2\Big) \\
    \leq C\Big(\|\J^{n+\frac12}\|_{L^2}^2+ \|\f^{n+\frac12}\|_{L^2}^2 \Big),
  \end{aligned}
\end{equation}
for $n=1,2,\dots,N-1$, where $C$ is a positive constant independent of $\tau$ and $h$.

\end{theorem}

\begin{proof}
By taking $\w_h=\widecheck\H_h^{n+\frac12}$ in \eqref{scheme-a} and
 $\v_h=\overline\u_h^{n+\frac12}$ in \eqref{scheme-b}, we get
\begin{align}
&\frac\mu{2\tau}\Big(\|\H_h^{n+1}\|_{L^2}^2-\|\H_h^n\|_{L^2}^2\Big)+\frac\mu{8\tau}\Big(\|\H_h^{n+1}-\H_h^n\|_{L^2}^2-
\|\H_h^n-\H_h^{n-1}\|_{L^2}^2\Big)   \nonumber\\
&\quad
+\frac\mu{8\tau}\|\H_h^{n+1}-2\H_h^n+\H_h^{n-1}\|_{L^2}^2
+\sigma^{-1}\|\nabla\times\widecheck\H_h^{n+\frac12}\|_{L^2}^2
+\sigma^{-1}\|\nabla\cdot\widecheck\H_h^{n+\frac12}\|_{L^2}^2 \nonumber\\
&\quad
-\mu\Big(\overline\u_h^{n+\frac12}\times\widetilde\H_h^{n+\frac12},\nabla\times\widecheck\H_h^{n+\frac12}\Big)   \nonumber\\
&=\sigma^{-1}\Big(\nabla\times \J^{n+\frac12},\widecheck\H_h^{n+\frac12} \Big)  ,
  \label{sta-H}
\end{align}
and
\begin{align}
&
\frac1{2\tau} \Big(\|\widehat\u_h^{n+1}\|_{L^2}^2-\|\u_h^n\|_{L^2}^2\Big)
+\nu\|\nabla\overline\u_h^{n+\frac12}\|_{L^2}^2
- \Big(p_h^n,\nabla\cdot\overline\u_h^{n+\frac12}\Big) \nonumber\\
&\quad
+\mu\Big(\widetilde\H_h^{n+\frac12}\times(\nabla\times\widecheck\H_h^{n+\frac12}),\overline\u_h^{n+\frac12} \Big)  \nonumber\\
     &=\Big(\f^{n+\frac12},\overline\u_h^{n+\frac12}\Big),
\end{align}
respectively,
where we have used the fact that $b(\widetilde\u_h^{n+\frac12},\overline\u_h^{n+\frac12},\overline\u_h^{n+\frac12})=0$, and
\begin{align*}
&\bigg(\frac{\H_h^{n+1}-\H_h^n}{\tau},\w_h \bigg) \\
&=
\frac1{2\tau}\Big(\|\H_h^{n+1}\|_{L^2}^2-\|\H_h^n\|_{L^2}^2\Big)+
\frac1{8\tau}\Big(\|\H_h^{n+1}-\H_h^n\|_{L^2}^2-\|\H_h^n-\H_h^{n-1}\|_{L^2}^2\Big)  \\
&\quad
+\frac1{8\tau}\|\H_h^{n+1}-2\H_h^n+\H_h^{n-1}\|_{L^2}^2 .
\end{align*}

In turn, a substitution of $\l_h=\u_h^{n+1}$ in \eqref{scheme-c} yields
\begin{align}
\frac{1}{2\tau}\Big(\|\u_h^{n+1}\|_{L^2}^2- \|\widehat\u_h^{n+1}\|_{L^2}^2
+\|\u_h^{n+1}-\widehat\u_h^{n+1}\|_{L^2}^2 \Big)=0,
\end{align}
where we have used the divergence-free condition $\eqref{scheme-d}$ for $q_h$ being $p_h^{n+1},p_h^n$.

Next, we choose $\l_h=\nabla_h p_h^n$ in \eqref{scheme-c} and obtain
\begin{align}
-\Big(\nabla\cdot\widehat\u_h^{n+1},p_h^n\Big)
=\frac{\tau}{4}\Big(\|\nabla_hp_h^{n+1}\|_{L^2}^2-\|\nabla_hp_h^n\|_{L^2}^2-\|\nabla_h(p_h^{n+1}-p_h^n)\|_{L^2}^2 \Big).
\end{align}
Furthermore, we get the following result from \eqref{scheme-c2}
\begin{align}
\frac{1}{4}\|\nabla_h(p_h^{n+1}-p_h^n)\|_{L^2}^2
=\frac{1}{\tau^2}\|\u_h^{n+1}-\widehat\u_h^{n+1}\|_{L^2}^2 . \label{sta-nablap}
\end{align}

Summing up \eqref{sta-H}-\eqref{sta-nablap} leads to
\begin{equation}
  \begin{aligned}
    &\frac\mu{2\tau} \Big(\|\H_h^{n+1}\|_{L^2}^2-\|\H_h^n\|_{L^2}^2\Big)
    +\frac\mu{8\tau}\Big(\|\H_h^{n+1}-\H_h^n\|_{L^2}^2-\|\H_h^n-\H_h^{n-1}\|_{L^2}^2\Big) \\
     &\quad
     +\frac\mu{8\tau}\|\H_h^{n+1}-2\H_h^n+\H_h^{n-1}\|_{L^2}^2
    +\sigma^{-1}\|\nabla\times\widecheck\H_h^{n+\frac12}\|_{L^2}^2
    +\sigma^{-1}\|\nabla\cdot\widecheck\H_h^{n+\frac12}\|_{L^2}^2  \\
    &\quad
    +\frac1{2\tau} \Big(\|\u_h^{n+1}\|_{L^2}^2-\|\u_h^n\|_{L^2}^2\Big)
    +\nu\|\nabla\overline\u_h^{n+\frac12}\|_{L^2}^2
    +\frac{\tau}{8}\Big(\|\nabla_hp_h^{n+1}\|_{L^2}^2-\|\nabla_hp_h^n\|_{L^2}^2 \Big) \\
    &\le
    \sigma^{-1}\Big(\nabla\times \J^{n+\frac12},\widecheck\H_h^{n+\frac12}\Big)
    +\Big(\f^{n+\frac12},\overline\u_h^{n+\frac12}\Big).
    \label{takeinnerpro}
  \end{aligned}
\end{equation}
For the right-hand side of \eqref{takeinnerpro}, we can easily see that
\begin{equation*}
  \begin{aligned}
    \sigma^{-1}\Big(\nabla\times \J^{n+\frac12},\widecheck\H_h^{n+\frac12}\Big)
    &=\sigma^{-1}\Big(\J^{n+\frac12},\nabla\times\widecheck\H_h^{n+\frac12}\Big) \\
     &\leq\frac{1}{4\sigma}\|\J^{n+\frac12}\|_{L^2}^2+\sigma^{-1}\|\nabla\times\widecheck\H_h^{n+\frac12}\|_{L^2}^2,
  \end{aligned}
\end{equation*}
and
\begin{equation*}
  \begin{aligned}
    \Big(\f^{n+\frac12},\overline\u_h^{n+\frac12}\Big)
    &\leq\frac1{4\varepsilon}\|\f^{n+\frac12}\|_{L^2}^2+\varepsilon\|\overline\u_h^{n+\frac12}\|_{L^2}^2
     \leq\frac1{4\varepsilon}\|\f^{n+\frac12}\|_{L^2}^2+
    \varepsilon\|\nabla\overline\u_h^{n+\frac12}\|_{L^2}^2,
  \end{aligned}
\end{equation*}
where $\varepsilon$ is an arbitrarily small constant. Substituting the above estimates into \eqref{takeinnerpro}, we get the desired result \eqref{est-stability} immediately. This completes the proof of Theorem \ref{stability}.
\end{proof}

\section{Optimal error estimates}\label{sec-error}

We present the proof of the existence and uniqueness of numerical solution and the optimal error estimates \eqref{est-error}-\eqref{est-error2} in Section \ref{sec-error}.

\subsection{Preliminary results}
We introduce several types of projections. Let $P_h: L^2(\Omega)\rightarrow M_h$ denote the $L^2$ projection which satisfies
\begin{align}
(v-P_hv, q_h)=0, \quad v\in L^2(\Omega), \, \, \, \forall q_h\in M_h.
\label{def-Ph}
\end{align}
For the sake of brevity, if $v$ (defined in \eqref{def-Ph}) is a vector function in ${\bf L}^2(\Omega)$, we still use $P_h$ to denote the $L^2$ projection over the finite element space ${\bf X}_h$.
Furthermore, let $(\R_h\u,R_hp)$ denote the Stokes projection of $(\u,p)\in {\bf{H}}^1_0(\Omega)\times L^2(\Omega)/\mathbb{R}$ satisfying
\begin{align}
&\nu(\nabla( \u- \R_h\u), \nabla\v_h)-(p-R_hp,\nabla\cdot\v_h)=0, &&\forall\, \v_h\in{\bf{X}}_h, \label{def-Rh1}\\
&(\nabla\cdot(\u-\R_h\u),q_h) = 0, &&\forall\, q_h\in M_h. \label{def-Rh2}
\end{align}
We also introduce the Maxwell projection operator $\Pi_h:$ $\ring{\bf H}^1(\Omega)\rightarrow {\bf S}_h$, by
\begin{align}
(\nabla\times(\H-\Pi_h\H), \nabla\times{\bm w}_h)+(\nabla\cdot(\H-\Pi_h\H), \nabla\cdot{\bm w}_h) =0,
\quad \H\in\ring{\bf H}^1(\Omega) , \forall {\bm w}_h \in {\bf S}_h .
\end{align}

For the above projections, the following estimates are recalled~\cite{GR1987,Thomee2006}.

\begin{lemma}
The following estimates are valid for the $L^2$ projection, Stokes projection, and Maxwell projection:
\begin{align}
&
\|P_hv\|_{W^{m,s}} \le C\|v\|_{W^{m,s}}, \label{Ih-est1} \\
&
\|v-P_hv\|_{L^2} \le Ch^{\ell+1}\|v\|_{H^{\ell+1}}, \label{Ih-est2}
\end{align}
for $m=0,1$, $0\le\ell\le r$, $1\le s\le \infty$,
and
\begin{align}
&\|\R_h\u\|_{W^{1,s}}+\|R_hp\|_{L^s} \le C (\|\u\|_{W^{1,s}}+\|p\|_{L^s}) , &&  \label{ph-infty} \\
&\|\u-\R_h\u\|_{L^s}+h\|\u-\R_h\u\|_{W^{1,s}} \le Ch^{\ell+1}(\|\u\|_{W^{\ell+1,s}}+\|p\|_{W^{\ell,s}}), \label{Ph-est1} \\
&\|p-R_hp\|_{L^s} \le Ch^{\ell}(\|\u\|_{W^{\ell+1,s}}+\|p\|_{W^{\ell,s}}), \label{Ph-est} \\
&
\|\partial_t(\u-\R_h\u)\|_{L^s} + h\|\partial_t(p-R_hp)\|_{L^s}
\le
Ch^{\ell+1}(\|\partial_t\u\|_{W^{\ell+1,s}}+\|\partial_t p\|_{W^{\ell,s}}),  \label{Ph-est-2}
\end{align}
for $0\le\ell\le r$, $1<s<\infty$,
and
\begin{align}
&
\|\H-\Pi_h\H\|_{L^2}+h\|\H-\Pi_h\H\|_{H^1} \le Ch^{\ell+1}\|\H\|_{H^{\ell+1}} ,
\label{Pih-est}
\end{align}
for $0\le\ell\le r$, where $C$ is a positive constant independent of  $h$.
\end{lemma}

Next, we recall two lemmas that will be frequently used in this paper.
\begin{lemma}[\cite{BS2002}]
Given $v_h$ in the finite element spaces ${\bf X}_h$, $M_h$, or ${\bf S}_h$, the following inverse inequality holds
\begin{align}
\|v_h\|_{W^{m,s}} \le Ch^{n-m+\frac{d}{s}-\frac{d}{q}}\|v_h\|_{W^{n,q}},
\label{inverse-1}
\end{align}
for $0\le n\le m\le 1$, $1\le q\le s\le \infty$, where $d$ denotes the dimension of the space and $C$ is a positive constant independent of  $h$.
\end{lemma}

\begin{lemma}\label{lem-nablah}
The discrete gradient operator $\nabla_h: M_h\rightarrow {\bf X}_h$ (defined in \eqref{def-nablah}) satisfies the following estimates
\begin{align}
&
\|\nabla_hq_h\|_{L^2} \le Ch^{-1}\|q_h\|_{L^2},
\label{inverse-2} \\
&
\|\nabla_hq_h\|_{L^3} \le Ch^{-1}\|q_h\|_{L^3},
\label{inverse-3}
\end{align}
for any $q_h\in M_h$, where $C$ is a positive constant independent of  $h$.
\end{lemma}
\begin{proof}
The estimate \eqref{inverse-2} follows immediately by substituting $\v_h=\nabla_hq_h$ into \eqref{def-nablah} and inverse inequality \eqref{inverse-1}.

It remains to prove  \eqref{inverse-3}. Given $q_h\in M_h$, it is easy to see that
\begin{align*}
(\nabla_hq_h,\v)
=(\nabla_hq_h,P_h\v)
=-(q_h,\nabla\cdot P_h\v)
&\le
\|q_h\|_{L^3}\|\nabla\cdot P_h\v\|_{L^\frac{3}{2}} \\
&\le
C\|q_h\|_{L^3}h^{-1}\|P_h\v\|_{L^\frac{3}{2}}
 \le
Ch^{-1}\|q_h\|_{L^3}\|\v\|_{L^\frac{3}{2}},
\end{align*}
for all $\v\in L^{\frac32}(\Omega)$.
Here, $P_h$ is the $L^2$ projection, which has a bounded extension to $L^p(\Omega)$ for $1\le p\le \infty$, with a bound independent of $h$; see \cite[Lemma 6.1]{Thomee2006}.
Then, using the duality between $L^3(\Omega)$ and $L^{\frac32}(\Omega)$, it is straightforward to derive \eqref{inverse-3}. The proof of Lemma~\ref{lem-nablah} is complete.
\end{proof}

\subsection{Error equations}
To establish error estimates for the scheme \eqref{scheme-a}-\eqref{scheme-d}, we introduce an intermediate
function $\widehat{\R_h\u^{n+1}}\in {\bm X}_h$, defined as
\begin{align}
&\bigg(\frac{\R_h\u^{n+1}-\widehat{\R_h\u^{n+1}}}{\tau},\l_h\bigg)-\frac12\Big(R_hp^{n+1}-R_hp^n,\nabla\cdot\l_h\Big)=0 ,  \quad \forall  \l_h\in{\bm X}_h , \label{Rh-hat}
\\
 &  \mbox{or equivalently,} \quad
\frac{\R_h\u^{n+1}-\widehat{\R_h\u^{n+1}}}{\tau}
=-\frac12\nabla_h\big(R_hp^{n+1}-R_hp^n\big) . \label{Rh-hat2}
\end{align}

With the intermediate function defined above and the projections introduced in the previous subsection, the MHD system \eqref{PDEa}-\eqref{PDEc} can be rewritten as follows:
\begin{align}
&
\mu\bigg(\frac{\Pi_h\H^{n+1}-\Pi_h\H^n}{\tau},\w_h\bigg)
+\sigma^{-1}\Big(\nabla\times\Pi_h\widecheck\H^{n+\frac12},\nabla\times\w_h \Big)
+\sigma^{-1}\Big(\nabla\cdot\Pi_h\widecheck\H^{n+\frac12},\nabla\cdot\w_h \Big)
\nonumber\\
&\quad\quad
-\mu\Big(\overline\u^{n+\frac12}\times\widetilde\H^{n+\frac12} ,\nabla\times\w_h\Big)
=\sigma^{-1}\Big(\nabla\times \J^{n+\frac12},\w_h\Big) +R_{\H}^{n+1}(\w_h),
\label{scheme-a-2} \\
&
\bigg(\frac{\widehat{\R_h\u^{n+1}}-\R_h\u^n}{\tau},\v_h\bigg)
+\nu\bigg(\nabla\Big(\frac{\widehat{\R_h\u^{n+1}} +\R_h\u^n}{2}\Big),\nabla\v_h\bigg)
+b\Big(\widetilde\u^{n+\frac12}, \overline\u^{n+\frac12},\v_h \Big)
\nonumber\\
&\quad\quad
-\Big(R_hp^n,\nabla\cdot\v_h\Big)
+\mu\Big(\widetilde\H^{n+\frac12}\times(\nabla\times\widecheck\H^{n+\frac12}),\v_h\Big)
\nonumber\\
&\quad=
\Big(\f^{n+\frac12},\v_h\Big)
+\bigg(\frac{\widehat{\R_h\u^{n+1}}-\R_h\u^{n+1}}{\tau},\v_h\bigg)
+\nu\bigg(\nabla\Big(\frac{\widehat{\R_h\u^{n+1}} - \R_h\u^{n+1}}{2}\Big),\nabla\v_h\bigg)
\nonumber\\
&\quad\quad
-\bigg(R_hp^n-\frac{R_hp^{n+1}+R_hp^n}{2},\nabla\cdot\v_h\bigg)
+R_{\u}^{n+1}(\v_h),
\label{scheme-b-2} \\
&
\Big(\nabla\cdot\u^{n+1},q_h\Big)=0,   \label{scheme-d-2}
\end{align}
for any $(\w_h,\v_h,q_h)\in({\bm S}_h,{\bm X}_h,M_h)$ and $n=1,2,\dots,N-1$, where we denote $\widehat \u^{n+1}:=\u^{n+1}$, $R_\H^{n+1}(\w_h)$ and $R_\u^{n+1}(\v_h)$ stand for the truncation errors satisfying
\begin{align}
&
R_{\H}^{n+1}(\w_h) \nonumber\\
&=\mu\bigg( \frac{\Pi_h\H^{n+1}-\Pi_h\H^n}{\tau}-\partial_t\H^{n+\frac12},\w_h\bigg)
+\sigma^{-1}\Big(\nabla\times(\Pi_h\widecheck\H^{n+\frac12}- \H^{n+\frac12}),\nabla\times\w_h \Big) \nonumber\\
&\quad
+\sigma^{-1}\Big(\nabla\cdot( \Pi_h\widecheck\H^{n+\frac12}- \H^{n+\frac12}),\nabla\cdot\w_h \Big)
-\mu\Big(\overline\u^{n+\frac12}\times\widetilde\H^{n+\frac12}-\u^{n+\frac12}\times\H^{n+\frac12} ,\nabla\times\w_h\Big) ,\\
&
R_{\u}^{n+1}(\v_h) \nonumber\\
&=\bigg(\frac{\R_h\u^{n+1}-\R_h\u^n}{\tau}-\partial_t\u^{n+\frac12},\v_h\bigg)
+\nu\bigg(\nabla\Big(\frac{\R_h\u^{n+1}+\R_h\u^n}{2}-\u^{n+\frac12}\Big),\nabla\v_h\bigg)
\nonumber\\
&\quad
+\bigg(b\Big(\widetilde\u^{n+\frac12}, \overline\u^{n+\frac12},\v_h \Big)-b\Big(\u^{n+\frac12}, \u^{n+\frac12},\v_h \Big) \bigg)
-\bigg(\frac{R_hp^{n+1}+R_hp^n}{2} -p^{n+\frac12},\nabla\cdot\v_h\bigg) \nonumber\\
&\quad
+\mu\Big(\widetilde\H^{n+\frac12}\times(\nabla\times\widecheck\H^{n+\frac12}) - \H^{n+\frac12}\times(\nabla\times\H^{n+\frac12}),\v_h\Big).
\end{align}

Utilizing the projection error estimates presented in the previous subsection, we only need to estimate the following error functions
\begin{align*}
&e_{\H}^n=\Pi_h\H^n-\H_h^n, \quad e_{\u}^n=\R_h\u^n-\u_h^n, \\
&\widehat e_{\u}^n=\widehat{\R_h\u^n}-\widehat\u_h^n, \quad\quad e_{p}^n=R_hp^n-p_h^n,
\end{align*}
for $n=1,2,\dots,N$.
From the system \eqref{Rh-hat}-\eqref{scheme-d-2} and the fully discrete numerical scheme \eqref{scheme-a}-\eqref{scheme-d}, we observe that the error functions $(e_{\H}^n,e_{\u}^n,\widehat e_{\u}^n,e_p^n)$ satisfy the following equations:
\begin{align}
&
\mu\bigg(\frac{e_{\H}^{n+1}-e_{\H}^n}{\tau},\w_h\bigg)
+\sigma^{-1}\Big(\nabla\times \widecheck e_{\H}^{n+\frac12},\nabla\times\w_h \Big)
+\sigma^{-1}\Big(\nabla\cdot \widecheck e_{\H}^{n+\frac12},\nabla\cdot\w_h \Big)
\nonumber\\
&\quad=
\mu\bigg\{\Big(\overline\u^{n+\frac12}\times\widetilde\H^{n+\frac12} ,\nabla\times\w_h\Big)
- \Big(\overline\u_h^{n+\frac12}\times\widetilde\H_h^{n+\frac12} ,\nabla\times\w_h\Big) \bigg\}
+R_{\H}^{n+1}(\w_h),
\label{error-a} \\
&
\bigg(\frac{\widehat e_{\u}^{n+1} - e_{\u}^n}{\tau},\v_h\bigg)
+\nu\Big(\nabla \overline e_{\u}^{n+\frac12},\nabla\v_h\Big)
-\Big(e_p^n,\nabla\cdot\v_h\Big)
\nonumber\\
&\quad=
\bigg(\frac{\widehat{\R_h\u^{n+1}}-\R_h\u^{n+1}}{\tau},\v_h\bigg)
+\nu\bigg(\nabla\Big(\frac{\widehat{\R_h\u^{n+1}} - \R_h\u^{n+1}}{2}\Big),\nabla\v_h\bigg)
\nonumber\\
&\quad\quad
-\bigg(R_hp^n-\frac{R_hp^{n+1}+R_hp^n}{2},\nabla\cdot\v_h\bigg)
-\bigg\{ b\Big(\widetilde\u^{n+\frac12}, \overline\u^{n+\frac12},\v_h \Big)
- b\Big(\widetilde\u_h^{n+\frac12}, \overline\u_h^{n+\frac12},\v_h \Big) \bigg\}
\nonumber\\
&\quad\quad
-\mu\bigg\{ \Big(\widetilde\H^{n+\frac12}\times(\nabla\times\widecheck\H^{n+\frac12}),\v_h\Big)
- \Big(\widetilde\H_h^{n+\frac12}\times(\nabla\times\widecheck\H_h^{n+\frac12}),\v_h\Big) \bigg\}
+R_{\u}^{n+1}(\v_h),
\label{error-b} \\
&
\bigg(\frac{e_{\u}^{n+1}-\widehat e_{\u}^{n+1}}{\tau},\l_h\bigg) - \frac12\Big(e_p^{n+1}-e_p^n,\nabla\cdot\l_h\Big)=0,  \label{error-c} \\
&
\Big(\nabla\cdot e_{\u}^{n+1},q_h\Big)=0,   \label{error-d}
\end{align}
for any $(\w_h,\v_h,\l_h,q_h)\in({\bm S}_h,{\bm X}_h,{\bm X}_h,M_h)$ and $n=1,2,\dots,N-1$.

\subsection{Proof of Theorem \ref{thm-error}}
In this subsection, we present a detailed proof of Theorem \ref{thm-error}. The following lemma will be used in the analysis.
\begin{lemma}\label{lem}
Under the regularity assumption \eqref{reg-asp},
the Stokes projection defined in \eqref{def-Rh1}-\eqref{def-Rh2} satisfies the following estimates:
\begin{align}
&
\|\nabla_hR_h\partial_t p\|_{L^3} \le C, \label{lem-est0} \\
&
\|\nabla(\nabla_hR_h\partial_t p)\|_{L^2}\le C, \label{lem-est1}
\end{align}
where $C$ is a positive constant independent of $h$.
\end{lemma}
\begin{proof}
By the regularity assumption \eqref{reg-asp} and the $L^3$ stability estimate of the $L^2$ projection, i.e., \eqref{Ih-est1}, we see that
\begin{align}
\|P_h\nabla\partial_t p\|_{L^3}
\le
C\|\nabla\partial_t p\|_{L^3}
\le
C. \label{lem-est0-1}
\end{align}
Since
\begin{align*}
(\v_h,P_h\nabla\partial_tp - \nabla_hP_h\partial_tp)
&=(\v_h,\nabla\partial_tp)+(\nabla\cdot\v_h,P_h\partial_tp)  \\
&=  - (\nabla\cdot\v_h,\partial_tp) +(\nabla\cdot\v_h,P_h\partial_tp)  \\
&\le \|\nabla\cdot\v_h\|_{L^{\frac{3}{2}}}\|\partial_t p-P_h\partial_t p\|_{L^3} \\
&\le Ch^{-1}\|\v_h\|_{L^{\frac32}}h\|\partial_tp\|_{W^{1,3}} \\
&\le C\|\v_h\|_{L^{\frac32}}\|\partial_tp\|_{W^{1,3}} ,
\end{align*}
for any $\v_h\in{\bf X}_h$, by the duality between $L^{\frac32}$ and $L^3$, we conclude that
\begin{align}
\|P_h\nabla\partial_tp-\nabla_hP_h\partial_tp\|_{L^3}\le C .  \label{lem-eq}
\end{align}
Consequently, with the help of the inverse inequality \eqref{inverse-3}, we obtain
\begin{align*}
\|\nabla_hR_h\partial_tp\|_{L^3}
&\le
\|\nabla_hR_h\partial_tp- \nabla_hP_h\partial_tp\|_{L^3}
+\|\nabla_hP_h\partial_tp-P_h\nabla\partial_tp\|_{L^3}
+\|P_h\nabla\partial_tp\|_{L^3}
\nonumber\\
&\le
\|\nabla_hR_h\partial_tp-\nabla_hP_h\partial_tp\|_{L^3}
+C
\nonumber\\
&\le
Ch^{-1}\|R_h\partial_tp-P_h\partial_tp\|_{L^3}
+C
\nonumber\\
&\le
Ch^{-1}\|R_h\partial_tp-\partial_tp\|_{L^3}
+Ch^{-1}\|\partial_tp-P_h\partial_tp\|_{L^3}
+C
\nonumber\\
&\le
Ch^{-1}h^2+Ch^{-1}h^2+C
\le
C,
\end{align*}
in which \eqref{Ih-est2} and the projection estimate \eqref{Ph-est-2} have been used in the second last inequality.

Inequality \eqref{lem-est1} could be proved in a similar manner. By the regularity assumption \eqref{reg-asp} and the $H^1$ stability estimate of the $L^2$ projection, we have
\begin{align}
\|\nabla P_h\nabla\partial_t p\|_{L^2}
\le
\|P_h\nabla\partial_t p\|_{H^1}
\le
C\|\nabla\partial_t p\|_{H^1}
\le
C. \label{lem-est1-1}
\end{align}
Using similar techniques in the derivation of~\eqref{lem-eq}, we get
\begin{align}
\|P_h\nabla\partial_tp-\nabla_hP_h\partial_tp\|_{L^2}\le Ch .
\end{align}
By the inverse inequalities \eqref{inverse-1}-\eqref{inverse-2}, it can be shown that
\begin{align}
&\|\nabla(\nabla_hR_h\partial_tp-P_h\nabla\partial_tp) \|_{L^2}  \nonumber\\
&\le
Ch^{-1}\|\nabla_hR_h\partial_tp-P_h\nabla\partial_tp \|_{L^2} \nonumber\\
&\le
Ch^{-1}\|\nabla_hR_h\partial_tp-\nabla_hP_h\partial_tp \|_{L^2}
+Ch^{-1}\|\nabla_hP_h\partial_tp -P_h\nabla\partial_tp \|_{L^2}
\nonumber\\
&\le
Ch^{-2}\|R_h\partial_tp - P_h\partial_tp \|_{L^2}
+Ch^{-1}h
 \nonumber\\
&\le
Ch^{-2}\|R_h\partial_tp - \partial_tp \|_{L^2}
+Ch^{-2}\|\partial_tp - P_h\partial_tp \|_{L^2} +C \nonumber\\
&\le
Ch^{-2}h^2+Ch^{-2}h^2 +C
\le
C, \label{lem-est1-2}
\end{align}
in which \eqref{Ih-est2} and \eqref{Ph-est-2} have been used again in the second to last inequality. Finally, by the triangle inequality and \eqref{lem-est1-1}-\eqref{lem-est1-2}, the estimate \eqref{lem-est1} follows immediately.
\end{proof}

\vspace{.05in}
Now we proceed with the proof of Theorem~\ref{thm-error}.

\begin{proof}[Proof of Theorem \ref{thm-error}.]
By Theorem \ref{stability}, the existence and uniqueness of numerical solution $(\H_h^n,\u_h^n,p_h^n)$, $n=2,3,\dots,N$, follows immediately since the scheme \eqref{scheme-a}-\eqref{scheme-d} is linearized and the corresponding homogeneous equations only admit zero solutions.

In the following, we present the analysis of the error equations \eqref{error-a}-\eqref{error-d} and then establish the optimal error estimates given in Theorem \ref{thm-error}. First of all, we
make the following induction assumption for the error functions at the previous time steps:
\begin{align}
\|e_\H^m\|_{L^2}+\|e_\u^m\|_{L^2} \le \tau^{\frac74}+h^{\frac{7}{4}} ,
\label{est-1}
\end{align}
for $m\le n$. Such an induction assumption will be recovered by the error estimate at the next time step $t_{n+1}$.

For $m=0,1$, \eqref{est-1} follows from Remark \ref{H1u1} immediately. The induction assumption \eqref{est-1} (for $m\le n$) yields
\begin{align}
\|\H_h^m\|_{W^{1,3}}
&\le
\|I_h\H^m\|_{W^{1,3}}+\|I_h\H^m-\Pi_h\H^m\|_{W^{1,3}} +\|e_\H^m\|_{W^{1,3}}
\nonumber\\
&\le
C\|\H^m\|_{W^{1,3}}
+Ch^{-\frac{d}{6}}\|I_h\H^m-\Pi_h\H^m\|_{H^1} +Ch^{-1-\frac{d}{6}}\|e_\H^m\|_{L^2}
\nonumber\\
&\le
C\|\H^m\|_{W^{1,3}}
+Ch^{-\frac{d}{6}}\|I_h\H^m-\H^m\|_{H^1}
+Ch^{-\frac{d}{6}}\|\H^m-\Pi_h\H^m\|_{H^1} \nonumber\\
&\quad
+Ch^{-1-\frac{d}{6}}\Big(\tau^{\frac74}+h^{\frac{7}{4}} \Big)
\nonumber\\
&\le
C\|\H^m\|_{W^{1,3}}
+Ch^{-\frac{d}{6}}h^2
+Ch^{-\frac{d}{6}}h^2
+Ch^{-1-\frac{d}{6}}\Big(\tau^{\frac74}+h^{\frac{7}{4}} \Big)
\quad(\mbox{by } \tau=\mathcal{O}(h))
\nonumber\\
&\le
K+1, \label{bound-H}  \\
\|\u_h^m\|_{L^\infty}
&\le
\|\R_h\u^m\|_{L^\infty} + \|e_\u^m\|_{L^\infty} \nonumber\\
&\le
\|\u^m\|_{W^{1,3}} + Ch^{-\frac{d}{2}}\|e_\u^m\|_{L^2}
\quad(\mbox{by \eqref{ph-infty}})
\nonumber\\
&\le
\|\u^m\|_{W^{1,3}} + Ch^{-\frac{d}{2}}\Big(\tau^{\frac74}+h^{\frac{7}{4}} \Big)
\quad(\mbox{by }\tau=\mathcal{O}(h))
\nonumber\\
&\le
K+1 , \label{bound-u}
\end{align}
for $h<h_0$, where $d=2,3,$ denotes the dimension of $\Omega$ and $h_0$ is a small positive constant. Here, $I_h$ denotes the standard Lagrange interpolation and its $W^{1,3}$ stability estimate has been used. Subsequently, we will establish the error estimate at $m=n+1$ and recover \eqref{est-1}.

{\it\underline{Step 1:} \, Estimate of \eqref{error-a}.} \,
Taking $\w_h=\widecheck e_{\H}^{n+\frac12}$ into \eqref{error-a} yields
\begin{align}
&\frac{\mu}{2\tau}\Big(\|e_{\H}^{n+1}\|_{L^2}^2 - \|e_{\H}^{n}\|_{L^2}^2\Big)
+\frac{\mu}{8\tau} \Big(\|e_{\H}^{n+1}-e_{\H}^n\|_{L^2}^2 - \|e_{\H}^{n}-e_{\H}^{n-1}\|_{L^2}^2 \Big) \nonumber\\
&\quad
+\sigma^{-1}\|\nabla\times\widecheck e_{\H}^{n+\frac12}\|_{L^2}^2
+\sigma^{-1}\|\nabla\cdot\widecheck e_{\H}^{n+\frac12}\|_{L^2}^2
\nonumber\\
&\le
\mu\bigg\{\Big(\overline\u^{n+\frac12}\times\widetilde\H^{n+\frac12} ,\nabla\times \widecheck e_{\H}^{n+\frac12} \Big)
- \Big(\overline\u_h^{n+\frac12}\times\widetilde\H_h^{n+\frac12} ,\nabla\times \widecheck e_{\H}^{n+\frac12}\Big) \bigg\}
+R_{\H}^{n+1}(\widecheck e_{\H}^{n+\frac12}) ,
\label{est-eH-1}
\end{align}
where we have used the identity
\begin{align}
\bigg(\frac{e_{\H}^{n+1}-e_{\H}^n}{\tau} , \widecheck e_{\H}^{n+\frac12} \bigg)
&=
\frac{1}{2\tau}\Big(\|e_{\H}^{n+1}\|_{L^2}^2 - \|e_{\H}^{n}\|_{L^2}^2\Big)
+\frac{1}{8\tau} \Big(\|e_{\H}^{n+1}-e_{\H}^n\|_{L^2}^2 - \|e_{\H}^{n}-e_{\H}^{n-1}\|_{L^2}^2 \Big)
\nonumber\\
&\quad
+\frac{1}{8\tau} \|e_{\H}^{n+1}-2e_{\H}^{n}+e_{\H}^{n-1}\|_{L^2}^2 .
\end{align}
By \eqref{reg-asp} and \eqref{Pih-est}, it can be shown that
\begin{align*}
R_{\H}^{n+1}(\widecheck e_{\H}^{n+\frac12})
\le
 C(\tau^2+h^{r+1})^2
+C\|\widecheck e_\H^{n+\frac12}\|_{L^2}^2
+\frac{1}{2\sigma}\|\nabla\times \widecheck e_{\H}^{n+\frac12}\|_{L^2}^2
+\frac{1}{2\sigma}\|\nabla\cdot \widecheck e_{\H}^{n+\frac12}\|_{L^2}^2 .
\end{align*}
Noticing that $\widehat\u^{n+1}:=\u^{n+1}$ and \eqref{Rh-hat2}, we obtain
\begin{align*}
&\mu\bigg\{\Big(\overline\u^{n+\frac12}\times\widetilde\H^{n+\frac12} ,\nabla\times \widecheck e_{\H}^{n+\frac12} \Big)
- \Big(\overline\u_h^{n+\frac12}\times\widetilde\H_h^{n+\frac12} ,\nabla\times \widecheck e_{\H}^{n+\frac12}\Big)  \bigg\}
\nonumber\\
&=
\mu\Big(\overline\u^{n+\frac12}\times\big(\widetilde\H^{n+\frac12} - \Pi_h\widetilde\H^{n+\frac12}\big) ,\nabla\times \widecheck e_{\H}^{n+\frac12} \Big)
+\mu\Big(\overline\u^{n+\frac12}\times \widetilde e_{\H}^{n+\frac12} ,\nabla\times \widecheck e_{\H}^{n+\frac12} \Big)
\nonumber\\
&\quad
+\mu\Big(\Big(\overline\u^{n+\frac12} - \frac{\R_h\u^{n+1}+\R_h\u^n}{2} \Big)\times\widetilde\H_h^{n+\frac12} ,\nabla\times \widecheck e_{\H}^{n+\frac12}\Big)
\nonumber\\
&\quad
+\mu\Big(\frac{R_h\u^{n+1}-\widehat{\R_h\u^{n+1}}}{2}\times\widetilde\H_h^{n+\frac12} ,\nabla\times \widecheck e_{\H}^{n+\frac12}\Big)
\nonumber\\
&\quad
+\mu\Big(\overline e_{\u}^{n+\frac12} \times\widetilde\H_h^{n+\frac12} ,\nabla\times \widecheck e_{\H}^{n+\frac12}\Big)
\nonumber\\
&\le
\mu\|\overline\u^{n+\frac12}\|_{L^\infty} \|\widetilde\H^{n+\frac12} - \Pi_h\widetilde\H^{n+\frac12}\|_{L^2}\|\nabla\times \widecheck e_{\H}^{n+\frac12} \|_{L^2}
+\mu\|\overline\u^{n+\frac12}\|_{L^\infty} \|\widetilde e_{\H}^{n+\frac12} \|_{L^2} \|\nabla\times \widecheck e_{\H}^{n+\frac12} \|_{L^2}
\nonumber\\
&\quad
+\mu\Big\|\frac{\u^{n+1}+\u^n}{2} - \frac{\R_h\u^{n+1}+\R_h\u^n}{2}\Big\|_{L^3} \|\widetilde\H_h^{n+\frac12}\|_{L^6} \|\nabla\times \widecheck e_{\H}^{n+\frac12}\|_{L^2}
\nonumber\\
&\quad
+\frac{\mu\tau}{4}\|\nabla_h(R_hp^{n+1}-R_hp^n)\|_{L^3}   \|\widetilde \H_h^{n+\frac12}\|_{L^6} \|\nabla\times \widecheck e_{\H}^{n+\frac12}\|_{L^2}
\nonumber\\
&\quad
+\mu\Big(\overline e_{\u}^{n+\frac12} \times\widetilde\H_h^{n+\frac12} ,\nabla\times \widecheck e_{\H}^{n+\frac12}\Big)
\nonumber\\
&\le
Ch^{2(r+1)}+\frac{1}{4\sigma}\|\nabla\times\widecheck e_\H^{n+\frac12}\|_{L^2}^2
+ C\|\widetilde e_\H^{n+\frac12}\|_{L^2}^2
+C\tau^4
+\mu\Big(\overline e_{\u}^{n+\frac12} \times\widetilde\H_h^{n+\frac12} ,\nabla\times \widecheck e_{\H}^{n+\frac12}\Big) ,
\end{align*}
where in the last inequality we have used the projection estimates \eqref{Ph-est1} and \eqref{Pih-est}, \eqref{bound-H}, Lemma \ref{lem} and the following inequality:
\begin{align*}
\frac{\mu\tau}{4}\|\nabla_h(R_hp^{n+1}-R_hp^n)\|_{L^3}
\le
C\tau^2 .
\end{align*}
With the above results, \eqref{est-eH-1} is reduced to
\begin{align}
&\frac{\mu}{2\tau}\Big(\|e_{\H}^{n+1}\|_{L^2}^2 - \|e_{\H}^{n}\|_{L^2}^2\Big)
+\frac{\mu}{8\tau} \Big(\|e_{\H}^{n+1}-e_{\H}^n\|_{L^2}^2 - \|e_{\H}^{n}-e_{\H}^{n-1}\|_{L^2}^2 \Big) \nonumber\\
&\quad+\frac{1}{4\sigma}\|\nabla\times\widecheck e_{\H}^{n+\frac12}\|_{L^2}^2
+\frac{1}{4\sigma}\|\nabla\cdot\widecheck e_{\H}^{n+\frac12}\|_{L^2}^2
\nonumber\\
&\le
C(\tau^2+h^{r+1})^2
+C\|\widecheck e_\H^{n+\frac12}\|_{L^2}^2
+ C\|\widetilde e_\H^{n+\frac12}\|_{L^2}^2
+\mu\Big(\overline e_{\u}^{n+\frac12} \times\widetilde\H_h^{n+\frac12} ,\nabla\times \widecheck e_{\H}^{n+\frac12}\Big) .
\label{est-eH-2}
\end{align}

{\it\underline{Step 2:} \, Estimate of \eqref{error-b}.} \,
Taking $\v_h=\overline e_\u^{n+\frac12}=\frac{1}{2}(\widehat e_\u^{n+1}+e_\u^n)$ into \eqref{error-b} leads to
\begin{align}
&
\frac{1}{2\tau}\Big(\|\widehat e_\u^{n+1}\|_{L^2}^2 -\| e_\u^n\|_{L^2}^2  \Big)
+\nu\|\nabla\overline e_\u^{n+\frac12}\|_{L^2}^2
-\big(e_p^n,\nabla\cdot\overline e_\u^{n+\frac12} \big)
\nonumber\\
&\le
\bigg(\frac{\widehat{\R_h\u^{n+1}}-\R_h\u^{n+1}}{\tau},\overline e_\u^{n+\frac12}\bigg)
+\nu\bigg(\nabla\Big(\frac{\widehat{\R_h\u^{n+1}} - \R_h\u^{n+1}}{2}\Big),\nabla\overline e_\u^{n+\frac12}\bigg)
\nonumber\\
&\quad
-\bigg(R_hp^n-\frac{R_hp^{n+1}+R_hp^n}{2},\nabla\cdot\overline e_\u^{n+\frac12}\bigg)
\nonumber\\
&\quad
-\bigg\{ b\Big(\widetilde\u^{n+\frac12}, \overline\u^{n+\frac12},\overline e_\u^{n+\frac12} \Big)
- b\Big(\widetilde\u_h^{n+\frac12}, \overline\u_h^{n+\frac12},\overline e_\u^{n+\frac12} \Big) \bigg\}
\nonumber\\
&\quad
-\mu\bigg\{  \Big(\widetilde\H^{n+\frac12}\times(\nabla\times\widecheck\H^{n+\frac12}),\overline e_\u^{n+\frac12}\Big)
-  \Big(\widetilde\H_h^{n+\frac12}\times(\nabla\times\widecheck\H_h^{n+\frac12}),\overline e_\u^{n+\frac12}\Big) \bigg\} \nonumber\\
&\quad
+R_{\u}^{n+1}(\overline e_\u^{n+\frac12})
\nonumber\\
&=:
\sum_{j=1}^6 \mathcal{I}_j .
\label{est-eu-1}
\end{align}
In the following, we estimate $\mathcal{I}_j$, $j=1,2,\dots,6$, respectively.
By using \eqref{Rh-hat}, we have
\begin{align*}
\mathcal{I}_1+\mathcal{I}_3
&=
-\bigg(\frac{R_hp^{n+1}+R_hp^n}{2} - \frac{R_hp^{n+1}+R_hp^n}{2}, \nabla\cdot \overline e_\u^{n+\frac12} \bigg)
=0.
\end{align*}
By \eqref{Rh-hat2}, $\mathcal{I}_2$ becomes
\begin{align*}
\mathcal{I}_2
&=
\frac{\nu\tau}{4}\bigg(\nabla \Big(\nabla_h(R_hp^{n+1}-R_hp^n) \Big),\nabla\overline e_\u^{n+\frac12}\bigg) \\
&\le
C\tau^2\|\nabla(\nabla_h(R_hp^{n+1}-R_hp^n))\|_{L^2}^2
+\varepsilon\|\nabla\overline e_\u^{n+\frac12}\|_{L^2}^2
\\
&\le
C\tau^4+\varepsilon\|\nabla\overline e_\u^{n+\frac12}\|_{L^2}^2 ,
\end{align*}
where we have used the second result in Lemma \ref{lem}.
By the definition of $b(\u,\v,\w)$ in \eqref{def-b}, we can rewrite $\mathcal{I}_4$ as
\begin{align*}
\mathcal{I}_4
&=
\frac12\bigg\{ \Big(\widetilde\u_h^{n+\frac12}\cdot\nabla\overline\u_h^{n+\frac12},\overline e_\u^{n+\frac12} \Big)
-\Big(\widetilde\u^{n+\frac12}\cdot\nabla\overline\u^{n+\frac12},\overline e_\u^{n+\frac12} \Big) \bigg\}
\nonumber\\
&\quad
-\frac12\bigg\{ \Big(\widetilde\u_h^{n+\frac12}\cdot\nabla\overline e_\u^{n+\frac12},\overline \u_h^{n+\frac12} \Big)
-\Big(\widetilde\u^{n+\frac12}\cdot\nabla\overline e_\u^{n+\frac12},\overline\u^{n+\frac12} \Big) \bigg\}
\nonumber\\
&=
-\frac12\bigg\{ \Big(\widetilde\u_h^{n+\frac12}\cdot\nabla\overline e_\u^{n+\frac12},\overline e_\u^{n+\frac12} \Big)
+\Big(\widetilde\u_h^{n+\frac12}\cdot\nabla(\overline\u^{n+\frac12}-\overline{\R_h\u}^{n+\frac12}),\overline e_\u^{n+\frac12} \Big)
\nonumber\\
&\quad\quad\quad
+\Big(\widetilde e_\u^{n+\frac12}\cdot\nabla\overline\u^{n+\frac12}, \overline e_\u^{n+\frac12}  \Big)
+\Big((\widetilde\u^{n+\frac12}-\widetilde{\R_h\u}^{n+\frac12})\cdot\nabla\overline\u^{n+\frac12},\overline e_\u^{n+\frac12}  \Big)\bigg\}
\nonumber\\
&\quad
+\frac12\bigg\{ \Big(\widetilde\u_h^{n+\frac12}\cdot\nabla\overline e_\u^{n+\frac12},\overline e_\u^{n+\frac12} \Big)
+\Big(\widetilde\u_h^{n+\frac12}\cdot\nabla \overline e_\u^{n+\frac12},\overline\u^{n+\frac12}-\overline{\R_h\u}^{n+\frac12} \Big)
\nonumber\\
&\quad\quad\quad
+\Big(\widetilde e_\u^{n+\frac12}\cdot\nabla \overline e_\u^{n+\frac12}  ,\overline\u^{n+\frac12}\Big)
+\Big((\widetilde\u^{n+\frac12}-\widetilde{\R_h\u}^{n+\frac12})\cdot\nabla\overline e_\u^{n+\frac12} ,\overline\u^{n+\frac12} \Big)\bigg\}
\nonumber\\
&=:
\frac12\sum_{k=1}^8\mathcal{I}_{4,k}.
\end{align*}
In the estimate of $\mathcal{I}_4$, the most difficult processing is the control of $\mathcal{I}_{4,2}$, for which an application of integration by parts implies that
\begin{align*}
\mathcal{I}_{4,2}
&=
\Big((\nabla\cdot\widetilde\u_h^{n+\frac12})(\overline\u^{n+\frac12}-\overline{\R_h\u}^{n+\frac12}),\overline e_\u^{n+\frac12} \Big)
+\Big(\widetilde\u_h^{n+\frac12}\cdot\nabla\overline e_\u^{n+\frac12},\overline\u^{n+\frac12}-\overline{\R_h\u}^{n+\frac12} \Big) \\
&=
\Big((\nabla\cdot\widetilde{\R_h\u}^{n+\frac12})(\overline\u^{n+\frac12}-\overline{\R_h\u}^{n+\frac12}),\overline e_\u^{n+\frac12} \Big)
-\Big((\nabla\cdot\widetilde e_\u^{n+\frac12})(\overline\u^{n+\frac12}-\overline{\R_h\u}^{n+\frac12}),\overline e_\u^{n+\frac12} \Big) \\
&\quad
+\Big(\widetilde\u_h^{n+\frac12}\cdot\nabla\overline e_\u^{n+\frac12},\overline\u^{n+\frac12}-\overline{\R_h\u}^{n+\frac12} \Big) \\
&\le
\|\nabla\cdot\widetilde{\R_h\u}^{n+\frac12}\|_{L^3} \|\overline\u^{n+\frac12}-\overline{\R_h\u}^{n+\frac12}\|_{L^2} \|\overline e_\u^{n+\frac12} \|_{L^6}  \\
&\quad
+\|\nabla\cdot\widetilde e_\u^{n+\frac12}\|_{L^3}  \|\overline\u^{n+\frac12}-\overline{\R_h\u}^{n+\frac12}\|_{L^2} \|\overline e_\u^{n+\frac12} \|_{L^6}  \\
&\quad
+\|\widetilde\u_h^{n+\frac12}\|_{L^\infty} \|\nabla\overline e_\u^{n+\frac12}\|_{L^2} \|\overline\u^{n+\frac12}-\overline{\R_h\u}^{n+\frac12} \|_{L^2} \\
&\le
C(\tau^2+h^{r+1})^2 +\varepsilon\|\nabla\overline e_\u^{n+\frac12}\|_{L^2}^2
+C\|\widetilde e_\u^{n+\frac12}\|_{L^2}^2,
\end{align*}
where we have used \eqref{ph-infty}, \eqref{Rh-hat2}, \eqref{bound-u},
\begin{align*}
\|\overline\u^{n+\frac12}-\overline{\R_h\u}^{n+\frac12}\|_{L^2}
&=
\bigg\|\frac{ \u^{n+1}+\u^n}{2}-\frac{ \widehat{\R_h\u^{n+1}}+\R_h\u^n}{2}
\bigg\|_{L^2} \quad (\mbox{here use }\widehat \u^{n+1}=\u^{n+1})
\nonumber\\
&\le
\bigg\|\frac{ \u^{n+1}+\u^n}{2}-\frac{ \R_h\u^{n+1}+\R_h\u^n}{2}
\bigg\|_{L^2}
+\bigg\|\frac{ \R_h\u^{n+1}-\widehat{\R_h\u^{n+1}}}{2}
\bigg\|_{L^2}
\nonumber\\
&\le
Ch^{r+1}+\frac{\tau}{4}  \|\nabla_h(R_hp^{n+1}-R_hp^n)\|_{L^2}
\quad (\mbox{by \eqref{lem-est0}})
\nonumber\\
&\le
Ch^{r+1}+C\tau^2
\end{align*}
and
\begin{align*}
\|\nabla\cdot\widetilde e_\u^{n+\frac12}\|_{L^3}  \|\overline\u^{n+\frac12}-\overline{\R_h\u}^{n+\frac12}\|_{L^2}
&\le
Ch^{-1-\frac{d}{6}}\|\widetilde e_\u^{n+\frac12}\|_{L^2}  \|\overline\u^{n+\frac12}-\overline{\R_h\u}^{n+\frac12}\|_{L^2} \\
&\le
Ch^{-1-\frac{d}{6}}(h^{r+1}+\tau^2) \|\widetilde e_\u^{n+\frac12}\|_{L^2}
\quad (\mbox{by }\tau=\mathcal{O}(h))  \\
&\le
C \|\widetilde e_\u^{n+\frac12}\|_{L^2}  .
\end{align*}
The estimate for other terms of $\mathcal{I}_4$ is straightforward. Clearly, $\mathcal{I}_{4,1}$ and $\mathcal{I}_{4,5}$ are cancelled.
By \eqref{Ph-est1} and \eqref{bound-u}, we obtain
\begin{align*}
\mathcal{I}_{4,3}+\mathcal{I}_{4,4}+\sum_{k=6}^8\mathcal{I}_{4,k}
&\le
C\|\widetilde e_\u^{n+\frac12}\|_{L^2}^2
+\varepsilon\|\nabla\overline e_\u^{n+\frac12}\|_{L^2}^2
+Ch^{2(r+1)} +C\tau^4 .
\end{align*}
Above all, we are led to
\begin{align*}
\mathcal{I}_4\le
C\|\widetilde e_\u^{n+\frac12}\|_{L^2}^2
+\varepsilon\|\nabla\overline e_\u^{n+\frac12}\|_{L^2}^2
+Ch^{2(r+1)} +C\tau^4.
\end{align*}
Similarly, we can rewrite $\mathcal{I}_5$ as
\begin{align*}
\mathcal{I}_5
&=
-\mu\bigg\{
\Big((\widetilde\H^{n+\frac12}-\Pi_h\widetilde\H^{n+\frac12})\times(\nabla\times\widecheck\H^{n+\frac12}) ,\overline e_\u^{n+\frac12} \Big)
+ \Big(\widetilde e_\H^{n+\frac12}\times(\nabla\times\widecheck\H^{n+\frac12}) ,\overline e_\u^{n+\frac12} \Big)
\nonumber\\
&\quad\quad\
+ \Big( \widetilde\H_h^{n+\frac12}\times(\nabla\times(\widecheck\H^{n+\frac12}-\Pi_h\widecheck\H^{n+\frac12})) ,\overline e_\u^{n+\frac12} \Big)
+ \Big( \widetilde\H_h^{n+\frac12}\times(\nabla\times\widecheck e_\H^{n+\frac12} ) ,\overline e_\u^{n+\frac12} \Big)
\bigg\}
\nonumber\\
&=:
\sum_{k=1}^4\mathcal{I}_{5,k}.
\end{align*}
By \eqref{Pih-est}, we have
\begin{align*}
\mathcal{I}_{5,1}
+\mathcal{I}_{5,2}
&\le
\mu\|\widetilde\H^{n+\frac12}-\Pi_h\widetilde\H^{n+\frac12}\|_{L^2} \|\nabla\times\widecheck\H^{n+\frac12}\|_{L^3} \|\overline e_\u^{n+\frac12}\|_{L^6}
\nonumber\\
&\quad
+\mu\|\widetilde e_\H^{n+\frac12}\|_{L^2}\|\nabla\times\widecheck\H^{n+\frac12}\|_{L^3}\|\overline e_\u^{n+\frac12} \|_{L^6}
\nonumber\\
&\le
Ch^{2(r+1)}+\varepsilon\|\nabla\overline e_\u^{n+\frac12}\|_{L^2}^2+C\|\widetilde e_\H^{n+\frac12}\|_{L^2}^2 .
\end{align*}
With an application of integration by parts, $\mathcal{I}_{5,3}$ becomes
\begin{align*}
\mathcal{I}_{5,3}
&=
\Big(\overline e_\u^{n+\frac12}\times\widetilde\H_h^{n+\frac12} ,\nabla\times(\widecheck\H^{n+\frac12}-\Pi_h\widecheck\H^{n+\frac12}) \Big)
\nonumber\\
&=
\Big(\nabla\times(\overline e_\u^{n+\frac12}\times\widetilde\H_h^{n+\frac12} ), \widecheck\H^{n+\frac12}-\Pi_h\widecheck\H^{n+\frac12} \Big) \\
&\le
\big\|\nabla\times(\overline e_\u^{n+\frac12}\times\widetilde\H_h^{n+\frac12} )\big\|_{L^2}\big\|\widecheck\H^{n+\frac12}-\Pi_h\widecheck\H^{n+\frac12}\|_{L^2} \\
&\le
\varepsilon\|\nabla\overline e_\u^{n+\frac12}\|_{L^2}^2+Ch^{2(r+1)} ,
\end{align*}
where we have used \eqref{bound-H} and
\begin{align*}
&\big\|\nabla\times\big(\overline e_\u^{n+\frac12}\times\widetilde\H_h^{n+\frac12} \big)\big\|_{L^2}
\nonumber\\
&=
\big\|\big(\nabla\cdot\widetilde\H_h^{n+\frac12}\big)\overline e_\u^{n+\frac12}
-\big(\nabla\cdot\overline e_\u^{n+\frac12}\big) \widetilde\H_h^{n+\frac12}
+\big(\widetilde\H_h^{n+\frac12}\cdot\nabla\big)\overline e_\u^{n+\frac12}
-\big(\overline e_\u^{n+\frac12} \cdot\nabla\big)\widetilde\H_h^{n+\frac12} \big\|_{L^2}
\nonumber\\
&\le
C\|\widetilde\H_h^{n+\frac12}\|_{W^{1,3}}\|\overline e_\u^{n+\frac12}\|_{L^6}
+C\|\nabla\overline e_\u^{n+\frac12}\|_{L^2}\|\widetilde\H_h^{n+\frac12}\|_{L^\infty} .
\end{align*}
Therefore, the following bound is available for $\mathcal{I}_{5}$:
\begin{align*}
\mathcal{I}_5
&\le
Ch^{2(r+1)}+\varepsilon\|\nabla\overline e_\u^{n+\frac12}\|_{L^2}^2+C\|\widetilde e_\H^{n+\frac12}\|_{L^2}^2
-\mu\Big(\widetilde\H_h^{n+\frac12}\times(\nabla\times\widecheck e_\H^{n+\frac12}) ,\overline e_\u^{n+\frac12} \Big) .
\end{align*}
A bound for the truncation error term $\mathcal{I}_6$ is based on \eqref{Ph-est1} and the regularity assumptions \eqref{reg-asp}:
\begin{align*}
\mathcal{I}_6
&\le
C(\tau^2+h^{r+1})^2
+\varepsilon\|\nabla\overline e_\u^{n+\frac12}\|_{L^2}^2.
\end{align*}
With the above estimates, we obtain the following result from \eqref{est-eu-1}
\begin{align}
&\frac{1}{2\tau}\Big(\|\widehat e_\u^{n+1}\|_{L^2}^2 -\| e_\u^n\|_{L^2}^2 \Big)
+\frac{\nu}{2}\|\nabla\overline e_\u^{n+\frac12}\|_{L^2}^2
-\big(e_p^n,\nabla\cdot\overline e_\u^{n+\frac12} \big)
\nonumber\\
&\le
C\|\widetilde e_\u^{n+\frac12}\|_{L^2}^2
+C(\tau^2+h^{r+1})^2
+C\|\widetilde e_\H^{n+\frac12}\|_{L^2}^2
-\mu\Big(\widetilde\H_h^{n+\frac12}\times(\nabla\times\widecheck e_\H^{n+\frac12}) ,\overline e_\u^{n+\frac12} \Big) .
\label{est-eu-2}
\end{align}

{\it\underline{Step 3:} \, Estimate of the term $-(e_p^n,\nabla\cdot\overline e_\u^{n+\frac12})$ in \eqref{est-eu-2}.} \,
We rewrite \eqref{error-c} as
\begin{align}
\frac{e_\u^{n+1}-\widehat e_\u^{n+1}}{\tau}
=-\frac{1}{2}\nabla_h(e_p^{n+1}-e_p^n) .
\label{error-d-2}
\end{align}
With the above equality and the fact that $\overline e_\u^{n+\frac12}=\frac12(\widehat e_\u^{n+1}+e_\u^n)$, we have
\begin{align}
-\big(e_p^n,\nabla\cdot\overline e_\u^{n+\frac12} \big)
&=
-\frac12\big(e_p^n,\nabla\cdot\widehat e_\u^{n+1} \big)
\quad (\mbox{by \eqref{error-d}})
\nonumber\\
&=
\frac12\big(\nabla_he_p^n, \widehat e_\u^{n+1} \big)
\nonumber\\
&=
\frac{\tau}{4} \big(\nabla_he_p^n ,\nabla_h(e_p^{n+1}-e_p^n) \big)
\quad (\mbox{by \eqref{error-d-2}})
\nonumber\\
&=
\frac{\tau}{8}\Big(\|\nabla_he_p^{n+1}\|_{L^2}^2 - \|\nabla_he_p^{n}\|_{L^2}^2 \Big)
-\frac{\tau}{8}\|\nabla_h(e_p^{n+1}-e_p^n)\|_{L^2}^2
\nonumber\\
&=
\frac{\tau}{8}\Big(\|\nabla_he_p^{n+1}\|_{L^2}^2 - \|\nabla_he_p^{n}\|_{L^2}^2 \Big)
-\frac{1}{2\tau}\|e_\u^{n+1}-\widehat e_\u^{n+1}\|_{L^2}^2 .
\label{est-ep-1}
\end{align}

{\it\underline{Step 4:}} \,
By taking $\l_h=e_\u^{n+1}$ into \eqref{error-c}, we arrive at
\begin{align}
\frac{1}{2\tau}\Big(\|e_\u^{n+1}\|_{L^2}^2 - \|\widehat e_\u^{n+1}\|_{L^2}^2
+\|e_\u^{n+1}-\widehat e_\u^{n+1}\|_{L^2}^2 \Big)
=0,
\label{est-eu-3}
\end{align}
in which \eqref{error-d} has been applied.

{\it\underline{Step 5:}}  \, A summation of \eqref{est-eH-2}, \eqref{est-eu-2}, \eqref{est-ep-1}, and \eqref{est-eu-3} leads to
\begin{align}
&\frac{\mu}{2\tau}\Big(\|e_{\H}^{n+1}\|_{L^2}^2 - \|e_{\H}^{n}\|_{L^2}^2\Big)
+\frac{\mu}{8\tau} \Big(\|e_{\H}^{n+1}-e_{\H}^n\|_{L^2}^2 - \|e_{\H}^{n}-e_{\H}^{n-1}\|_{L^2}^2 \Big) \nonumber\\
&\quad
+\frac{1}{4\sigma}\|\nabla\times\widecheck e_{\H}^{n+\frac12}\|_{L^2}^2
+\frac{1}{4\sigma}\|\nabla\cdot\widecheck e_{\H}^{n+\frac12}\|_{L^2}^2
\nonumber\\
&\quad
+\frac{1}{2\tau}\Big(\|e_\u^{n+1}\|_{L^2}^2-\|e_\u^n\|_{L^2}^2 \Big)
+\frac{\nu}{2}\|\nabla\overline e_\u^{n+\frac12}\|_{L^2}^2
+\frac{\tau}{8}\Big(\|\nabla_h e_p^{n+1}\|_{L^2}^2 - \|\nabla_h e_p^{n}\|_{L^2}^2 \Big)
\nonumber\\
&\le
C\|\widecheck e_\H^{n+\frac12}\|_{L^2}^2
+C\|\widetilde e_\H^{n+\frac12}\|_{L^2}^2
+C\|\widetilde e_\u^{n+\frac12}\|_{L^2}^2
+C(\tau^2+h^{r+1})^2 .
\end{align}
An application of discrete Gronwall's inequality indicates that there exists a positive constant $\tau_0$ such that
\begin{align}\label{est-final}
&\|e_\H^{n+1}\|_{L^2}^2 + \tau\sum_{m=1}^n\|\nabla\times\widecheck e_\H^{m+\frac12}\|_{L^2}^2
+ \tau\sum_{m=1}^n\|\nabla\cdot\widecheck e_\H^{m+\frac12}\|_{L^2}^2
\nonumber\\
&\quad
+\|e_\u^{n+1}\|_{L^2}^2 + \tau\sum_{m=1}^n\|\nabla\overline e_\u^{m+\frac12}\|_{L^2}^2
+\tau^2\|\nabla_h e_p^{n+1}\|_{L^2}^2
\le
C(\tau^2+h^{r+1})^2 ,
\end{align}
if $\tau<\tau_0$. By applying the Cauchy's inequality
\begin{align}
\|\nabla\times\widecheck e_\H^{n+\frac12}\|_{L^2}^2
&\ge
\frac38\|\nabla\times e_\H^{n+1}\|_{L^2}^2 - \frac18\|\nabla\times e_\H^{n-1}\|_{L^2}^2 , \\
\|\nabla\cdot\widecheck e_\H^{n+\frac12}\|_{L^2}^2
&\ge
\frac38\|\nabla\cdot e_\H^{n+1}\|_{L^2}^2 - \frac18\|\nabla\cdot e_\H^{n-1}\|_{L^2}^2 ,
\end{align}
we further get the following result from \eqref{est-final}
\begin{align}
&\|e_\H^{n+1}\|_{L^2}^2 + \tau\sum_{m=1}^n\|\nabla\times e_\H^{m+1}\|_{L^2}^2
+\tau\sum_{m=1}^n\|\nabla\cdot e_\H^{m+1}\|_{L^2}^2
\nonumber\\
&\quad
+\|e_\u^{n+1}\|_{L^2}^2 + \tau\sum_{m=1}^n\|\nabla\overline e_\u^{m+\frac12}\|_{L^2}^2
+\tau^2\|\nabla_h e_p^{n+1}\|_{L^2}^2
\le
C(\tau^2+h^{r+1})^2 .
\end{align}
The above estimate implies that the induction assumption \eqref{est-1} could be recovered at $m=n+1$, if $\tau$ and $h$ are sufficiently small.
Thus the mathematical induction is closed. By the projection estimates \eqref{Ph-est1} and \eqref{Pih-est}, the error estimates \eqref{est-error}-\eqref{est-error2} in Theorem \ref{thm-error} follow immediately.
\end{proof}

%
%
%

\section{Numerical examples}\label{sec:numerical results}
\setcounter{equation}{0}

In this section, we present several numerical examples to illustrate our theoretical results in Theorems \ref{thm-error} and \ref{stability}.
For the sake of simplicity, numerical results are tested for two-dimensional problems in a unit square domain.

\begin{example}
\rm
First, we consider the MHD equations
\begin{align}
&    \mu\partial_t\H+\sigma^{-1}\nabla\times(\nabla\times\H)
        -\mu\nabla\times(\u\times\H)={\bm g},
        \label{PDEa-2}\\
&
    \partial_t\u+\u\cdot\nabla\u-\nu\Delta\u+\nabla p
        =\f-\mu\H\times(\nabla\times\H),
        \label{PDEb-2}\\
&
    \nabla\cdot\H=0,\quad \nabla\cdot\u=0,\label{PDEc-2}
\end{align}
in $\Omega=[0,1]\times[0,1]$, with the initial and boundary conditions \eqref{initial data-1}-\eqref{BC-1}, where the source terms ${\bm g}$ and $\f$ are chosen correspondingly to the exact solutions
\begin{equation} \label{exactsolution}
  \begin{aligned}
    \u&=t^4\left(
    \begin{aligned}
      \sin^2(\pi x)\sin(2\pi y) \\
      -\sin(2\pi x)\sin^2(\pi y)
    \end{aligned}
    \right),  \\
     \H&=t^4\left(
    \begin{aligned}
      -\sin(2\pi y)\cos(2\pi x) \\
      \sin(2\pi x)\cos(2\pi y)
    \end{aligned}
    \right), \\
    p&=t^4\sin(2\pi x)\sin(2\pi y).
  \end{aligned}
\end{equation}
For simplicity, all the coefficients $\nu, \sigma, \mu$ in~\eqref{PDEa-2}-\eqref{PDEc-2} are chosen to be $1$, and we take the final time as $T=1$. Note that the above exact solutions $\u$ and $\H$ satisfy the divergence-free conditions.

We solve the MHD system \eqref{PDEa-2}-\eqref{PDEc-2} by the modified Crank--Nicolson FEM scheme \eqref{scheme-a}-\eqref{scheme-d} with a quadratic finite element approximation for
$\H$ and $\u$, and a linear finite element approximation for $p$.
To investigate the convergence rate in time, we first choose $\tau=T/N$ with $N=10,20,40,80$, with a sufficiently small spatial mesh size $h=1/100$ such that the spatial discretization error can be relatively negligible.
We present the numerical results at time $T=1$ in Table \ref{table1}, which indicate that the proposed scheme has second-order convergence in time.
\begin{table}[htb]
\begin{center}
\begin{tabular}{c|c|c|c|c}
\hline
\hline
   $\tau$ &$\|\u^N-\u_h^N\|_{L^2 }$  &Order &$\|\H^N-\H_h^N\|_{L^2 }$  &Order \\ \hline
$1/10$   &$9.636\times10^{-3}$  &         &$2.885\times10^{-2}$  &  \\ \hline
$1/20$   &$2.390\times10^{-3}$  &2.01  &$7.504\times10^{-3}$  &1.94   \\ \hline
$1/40$   &$5.741\times10^{-4}$  &2.06  &$1.912\times10^{-4}$  &1.97   \\ \hline
$1/80$   &$1.409\times10^{-4}$  &2.03  &$4.823\times10^{-4}$  &1.99   \\ \hline
\hline
\end{tabular}
\vspace{.1in}
\caption{Temporal convergence at $T=1$}
\label{table1}
\end{center}
\end{table}

Then we solve the problem \eqref{PDEa-2}-\eqref{PDEc-2} by the modified Crank--Nicolson FEM scheme \eqref{scheme-a}-\eqref{scheme-d} with a sufficiently small temporal step size $\tau=1/2000$, to focus on the spatial convergence rate. Again, a quadratic finite element approximation for $\H$ and $\u$ is applied, combined with a linear finite element approximation for $p$. Here, we take $h=1/10,1/20,1/40,1/80$. Numerical results at $T=1$ are presented in Table \ref{tableh}. It is observed that the errors in $L^2$-norm are proportional to $h^3$, which are consistent with the theoretical analysis in Theorem \ref{thm-error}.
\begin{table}[htb]
\begin{center}
\begin{tabular}{c|c|c|c|c}
\hline
\hline
    $h$    &$\|\u^N-\u_h^N\|_{L^2 }$  &Order &$\|\H^N-\H_h^N\|_{L^2 }$  &Order  \\ \hline
$1/10$   &$1.510\times10^{-3}$  &         &$2.723\times10^{-3}$  &  \\ \hline
$1/20$   &$1.906\times10^{-4}$  &2.99  &$3.433\times10^{-4}$  &2.99   \\ \hline
$1/40$   &$2.392\times10^{-5}$  &2.99  &$4.313\times10^{-5}$  &2.99   \\ \hline
$1/80$   &$3.008\times10^{-6}$  &2.99  &$5.480\times10^{-6}$  &2.98  \\ \hline
\hline
\end{tabular}
\vspace{.1in}
\caption{Spatial convergence at $T=1$}
\label{tableh}
\end{center}
\end{table}

\end{example}

\begin{example}
\rm
Second, we test the energy stability of the proposed scheme by solving the problem \eqref{PDEa}-\eqref{BC-1} in $\Omega=[0,1]\times[0,1]$ with $\J=\f={\bm 0}$ ($\J$ denotes a scalar function in $\mathbb{R}^2$) and $T=1$. Here, all the coefficients $\nu, \sigma, \mu$ in \eqref{PDEa}-\eqref{BC-1} are chosen to be $1$ and the initial values are chosen as:
\begin{equation} \label{initial-date-2}
  \begin{aligned}
    \u_0&=\left(
    \begin{aligned}
      \sin^2(\pi x)\sin(2\pi y) \\
      -\sin(2\pi x)\sin^2(\pi y)
    \end{aligned}
    \right),  \\
     \H_0&=\left(
    \begin{aligned}
      -\sin(2\pi y)\cos(2\pi x) \\
      \sin(2\pi x)\cos(2\pi y)
    \end{aligned}
    \right), \\
    p_0&=\sin(2\pi x)\sin(2\pi y).
  \end{aligned}
\end{equation}

We solve the problem by the proposed scheme \eqref{scheme-a}-\eqref{scheme-d} with a quadratic finite element approximation for $\H$ and $\u$, combined with a linear finite element approximation for $p$. The time step size and spatial mesh size are chosen as $\tau=10$ and $h=1/50$, respectively. We define the energy function as $E_h^n:=\|\u_h^n\|_{L^2}^2+\|\H_h^n\|_{L^2}^2+\frac14\|\H_h^n-\H_h^{n-1}\|_{L^2}^2+\frac{\tau^2}{4}\|\nabla_hp_h^n\|_{L^2}^2$. The energy evolution curve, up to the final time $T=1000$, is displayed in Figure~\ref{fig2}, which clearly indicates the energy dissipation property, consistent with the theoretical result in Theorem \ref{stability}.
\begin{figure}[htpb]
  \centering
  \includegraphics[width=3.5in]{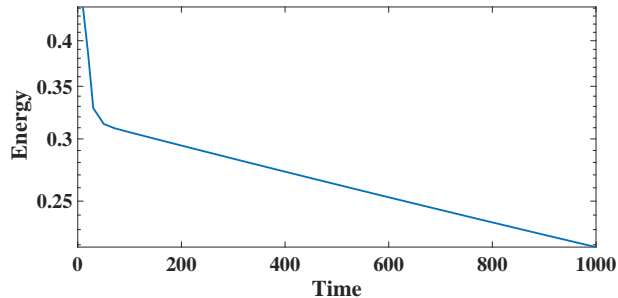}\\
  \caption{Energy of the MHD system at each time level}
\label{fig2}
\end{figure}

\end{example}

\section{Conclusion}   \label{sec:conclusion}

In this paper, we propose a decoupled and temporally second-order accurate, finite element method for the incompressible magnetohydrodynamic equations~\eqref{PDEa}-\eqref{PDEc}.
The primary difficulties are associated with the nonlinear and coupled nature of the problem.
In this work, a modified Crank--Nicolson method is used for the temporal discretization, and appropriate semi-implicit treatments are adopted for the approximation of the fluid convection term and two coupled terms. Then a linear system with variable coefficients is presented and its unique solvability is theoretically proved by the fact that the corresponding homogeneous equations only admit zero solutions. One prominent advantage of the scheme is associated with a decoupling approach in the Stokes solver, which computes an intermediate velocity field based on the pressure gradient at the previous time step, and then enforces the incompressibility constraint via the Helmholtz decomposition of the intermediate velocity field. As a result, this decoupling approach greatly reduces the computation of the MHD system. Furthermore, the energy stability analysis and optimal error estimates in the discrete $L^\infty(0,T;L^2)$ norm are provided for the scheme, in which the decoupled Stokes solver needs to be carefully estimated. Several numerical examples are presented to demonstrate the robustness and accuracy of the proposed scheme. The extension of the energy stable projection methods and its optimal-order error estimates to two-phase MHD models will be investigated in the future.



\vspace{.05in}
{\bf Funding.}
The research of C.~Wang was supported in part by NSF DMS-2012669.
The research of J.~Wang was supported in part by NSFC-U1930402 and NSFC-12071020.
The research of Z.~Xia was supported in part by NSFC-11871139.
The research of L.~Xu was supported in part by NSFC-11771068 and NSFC-12071060.

	\bibliographystyle{plain}
	\bibliography{MHD}




\end{document}